\documentclass[11pt,xcolor=dvipsnames,svgnames,table,reqno]{amsart}

\usepackage[utf8]{inputenc}
\usepackage{graphicx}
\graphicspath{ {./images/} }
\usepackage{hyperref}
\usepackage{amssymb} 
\usepackage{comment}
\usepackage{amsthm} 
\usepackage{mathtools}
\usepackage{nccmath}
\usepackage{physics} 
\usepackage{tipa} 
\usepackage{yfonts}
\usepackage{enumerate} 
\usepackage[T1]{fontenc} 
\usepackage[inline]{enumitem} 
\usepackage{xspace} 
\usepackage{commath}
\usepackage{tikz-cd}
\usepackage{color} 
\usepackage{pgfplots}
\pgfplotsset{compat=1.18}

\usepackage[left=1.5in, right=1.5in, bottom=1.25in,
height=8in]{geometry}
\usepackage{wrapfig}

\newcommand{\tope}[1]{\mathrm{h}_{\mathrm{top}}(#1)}

\usepackage{tikz}
\usetikzlibrary{shapes,shadows,arrows}
\usepackage{tikz-cd}
\usetikzlibrary{cd}

\newtheorem{definition}{Definition}
\newtheorem{theorem}{Theorem}

\newtheorem{lemma}{Lemma}
\newtheorem{corollary}{Corollary}
\newtheorem{example}{Example}
\newtheorem{remark}{Remark}
\newtheorem{proposition}{Proposition}
\newtheorem*{customtheorem}{Theorem A} 

\title{\textbf{BARCODE ENTROPY AND WRAPPED FLOER HOMOLOGY}}
\author{Rafael Fernandes}

\address{Department of Mathematics, UC Santa Cruz, Santa
  Cruz, CA 95064, USA} \email{rfernan9@ucsc.edu}

\begin{document}

\maketitle
\begin{abstract}
    In this paper, we explore the interplay between barcode and topological entropies. Wrapped Floer homology barcode entropy is the exponential growth of not-to-short bars in the persistence module associated with the filtered wrapped Floer homology. We prove that the barcode entropy is an invariant of the contact boundary of a Liouville domain, i.e., does not depend on the filling, and it bounds from below the topological entropy of the Reeb flow. We make no assumptions on the first Chern class of the Liouville domain or the contact form. 
\end{abstract}
\tableofcontents

\section{Introduction and main results}

\subsection{Introduction}
The aim of this paper it to further explore the relations between the Floer homology and the dynamics of the underline system. Our investigation centers on the interplay between barcode and topological entropies. More precisely, we define barcode entropy of the wrapped Floer homology and prove that it is bounded from above by the topological entropy of the underlining Reeb flow. 

The barcode entropy of a persistence module is the exponential growth of the number of not-too-short bars. This number is an invariant of the persistence module. Within the framework of Floer theory, different flavors of Floer homology have a natural filtration given by the action and can be viewed as a persistence module, thus having an associated barcode. The barcode entropy turns out to be relevant in this case. In \cite{cineli2021topological}, barcode entropy was introduced in the case of a sequence of iterates of a compactly supported Hamiltonian diffeomorphism $\varphi:M \rightarrow M$, denoted by $\hbar(\varphi)$. It was shown there that the barcode entropy cannot exceed the topological entropy, i.e., $\hbar(\varphi) \leq \tope{\varphi}$ (Theorem A). Furthermore, in the presence of a compact hyperbolic invariant set $K$ for $\varphi$, we have $\hbar(\varphi|_K) \leq \tope{\varphi}$ (Theorem B). These results imply in particular that this invariant is finite and non-trivial. As a consequence of the results of Katok in \cite{katok1980lyapunov}, it was also proved in \cite{cineli2021topological}  that $\hbar(\varphi) = \tope{\varphi}$ (Corollary C) when $M$ is a surface (this equality does not hold in general in higher dimension, as shown by Çineli in \cite{cineli2023generalized}). For the case of geodesic flows, similar results were obtained in \cite{ginzburg2022barcode} (Theorems A,B, and Corollary C in dimension 3). For the case of Reeb flows, the barcode entropy $\hbar(\alpha)$ of a contact form $\alpha$ on a strongly fillable manifold was introduced in \cite{fender2023barcode}, where it was shown that, under certain conditions the entropy does not depend on the filling, and that $\hbar(\alpha) \leq \tope{\varphi^t_{\alpha}} = \tope{\alpha}$, where $\varphi^t_{\alpha}$ is the Reeb flow of $\alpha$. Recently, a contact version of Theorem B was established for Reeb flow in \cite{cineli2024barcode}. Namely, $\tope{\varphi^t_{\alpha}|_K}\leq \hbar(\alpha)$, for a hyperbolic invariant set $K$. What we do on this paper can be thought as the next natural follow up question after this series of works. We define the barcode entropy of the persistence module given by the wrapped Floer homology of a Liouville domain together with two exact asymptotically conical Lagrangians, denoted by $\hbar(M,L_0 \rightarrow L_1)$, and prove a version of theorem A, i.e., that $\hbar(M,L_0 \rightarrow L_1) \leq \tope{\alpha}$; see Section \ref{mainresults} for the definitions.

Connections between topological entropy of geodesic/Reeb flows and global topological/geometric properties of the phase/configuration space is a well studied subject. See for example \cite{dinaburg1971relations,katok1982entropy,paternain2012geodesic,alves2016cylindrical,frauenfelder2006fiberwise, macarini2011positive,meiwes2018rabinowitz}, and references therein. What distinguishes the approach based on barcode entropy is that it is not tied up to global topological features. For example, the barcode entropy can be positive independently of the growth of $\pi_1(M)$ or $H_*(\Lambda)$, e.g., when $M$ is a sphere or torus and the underling flow admit a hyperbolic set with positive topological entropy.

\subsection{Main results} \label{mainresults}
Let $(M,\lambda)$ be a Liouville domain, and denote by $\alpha=\lambda|_{\Sigma}$ the restriction of $\lambda$ to the boundary $\Sigma = \partial M$. Consider a pair of exact asymptotically conical Lagrangians $L_0$ and $L_1$.

We denote the filtered wrapped Floer homology of $(M,L_0 \rightarrow L_1)$ on the interval $(-\infty,t)$, for $t>0$, by $HW^t(M,L_0 \rightarrow L_1)$. We point out that the grading does not play a role in what follows, and therefore we make no assumptions on the first Chern class $c_1(TW)$. We may suppress the Lagrangians from the notation for simplicity when they are well understood.

The family of vector spaces $\{V_a = HW^a(M)\}$, where $V_a = \{0\}$ for $a\leq 0$, together with the "inclusion" maps $HW^a(M) \rightarrow HW^b(M)$, for $a<b$, induced by the inclusion map of the complexes form a persistence module, and we denote its barcode by $B(M,L_0 \rightarrow L_1)$, or shortly, $B(H)$. For $\epsilon >0$, we denote by $\text{\textcrb}_{\epsilon}(M,L_0 \rightarrow L_1,t)$, or shortly $\text{\textcrb}_{\epsilon}(M,t)$ the number of bars in $B(M)$ with length greater or equal to $\epsilon$ with left end point at most $t$. This is a increasing function of $t$ and $ 1/\epsilon$ and locally constant as a function of $t$ on the complement of $S(\alpha,\Lambda_0 \rightarrow \Lambda_1) \cup \{0\}$, where $S(\alpha,\Lambda_0 \rightarrow \Lambda_1)$ is the set of lengths of Reeb chords connecting $\Lambda_0 = \partial L_0$ and $\Lambda_1 = \partial L_1$. The main invariant studied at this paper is defined as follows. (Here $\log^+(x) = \log(x)$, for $x > 0$, and $\log^+(0)=0$.)

\begin{definition}
      The wrapped Floer homology \textbf{$\epsilon$-barcode entropy} of $M$, for $\epsilon>0$, is defined by 
    $$\hbar^{HW}_{\epsilon}(M,L_0 \rightarrow L_1) =\limsup_{t \rightarrow \infty} \frac{\log^{+}(\text{\textcrb}_{\epsilon}(M,L_0 \rightarrow L_1,t))}{t},$$
    and the wrapped Floer homology \textbf{barcode entropy} of $M$ is
    $$\hbar^{HW}(M,L_0 \rightarrow L_1) =\lim_{\epsilon \rightarrow 0} \hbar^{HW}_{\epsilon}(M,L_0 \rightarrow L_1).$$
\end{definition}
The barcode entropy measures the exponential growth of the number of not-too-short bars. We notice that this definition makes sense in a broader setting, i.e., for any persistence module in the sense of Section \ref{barcodes}, we can define its barcode entropy.
 
The main theorems of this paper can be thought as the natural step after the series of papers \cite{cineli2021topological,ginzburg2022barcode,meiwes2018rabinowitz,fender2023barcode}, on the quest to relate filtered Floer type homology and topological entropy of the underlining dynamical system. To state it precisely, we denote $\tope{\alpha}$ the topological entropy of the Reeb flow $\varphi^t_{\alpha}$, where $R_{\alpha}$ is the Reeb vector field for $\alpha$. 

For $(\Sigma,\alpha, \Lambda_0 \rightarrow \Lambda_1)$, where $(\Sigma,\alpha)$ is a compact contact manifold, and $\Lambda_0$ and $\Lambda_1$ are Legendrian submanifolds, we say that $(M,\lambda,L_0 \rightarrow L_1)$, is a filling, where $(M,\lambda)$ is a Liouville domain, and $L_0$ and $L_1$ connected exact asymptotically conical Lagrangians if $\partial M = \Sigma$, $\lambda|_{\Sigma} = \alpha$, and $\Lambda_0 = \partial L_0$ and $\Lambda_1 = \partial L_1$, with $L_0 \pitchfork L_1$.

\begin{theorem}\label{thm:independence}
    Let $(M,\lambda,L_0 \rightarrow L_1)$ and $(N,\eta, K_0 \rightarrow K_1)$ two fillings for $(\Lambda,\alpha,\Lambda_0 \rightarrow \Lambda_1)$, then 
    \begin{equation} \label{equationindependence}
        \hbar^{HW}(M,L_0 \rightarrow L_1) = \hbar^{HW}(N,K_0 \rightarrow K_1).
    \end{equation}
\end{theorem} 

We then define $\hbar^{HW}(\alpha,\Lambda_0 \rightarrow \Lambda_1) =\hbar^{HW}(M,L_0 \rightarrow L_1)$, where $(M,\lambda, L_0 \rightarrow L_1)$ is any filling of $(\Lambda,\alpha,\Lambda_0 \rightarrow \Lambda_1)$.
\begin{remark}
    We point out that in the case of symplectic homology, the barcode entropy of compact contact strongly fillable manifold $(\Sigma,\alpha)$ does not depend on the filling at least when $c_1(TM) = 0$; see \cite[section 4.2]{fender2023barcode}. We tend to believe that the condition on the first Chern class of $TM$ is unnecessary. 
\end{remark}

\begin{customtheorem} \label{theoremA} Let $(M,\lambda)$ be a Liouville domain, and $L_0$, and $L_1$ be connected exact asymptotically conical Lagrangians in $M$. Then
\begin{align} 
    \label{equationentropies}
    \hbar^{HW}(M,L_0 \rightarrow L_1) \leq \tope{\alpha}. 
\end{align}
\end{customtheorem}
\begin{remark}
    It is worth nothing that from the work of Meiwes in, for example \cite{meiwes2018rabinowitz}, the wrapped Floer homology barcode entropy is not a trivial invariant of the Liouville domain. Indeed, the barcode entropy is bounded by below by the symplectic growth of $HW$ (the exponential growth of the number of infinity bars in the persistence module $a \rightarrow HW^a(M,L_0 \rightarrow L_1)$), which is positive in some examples; see \cite[Proposition 2.5]{meiwes2018rabinowitz}.
\end{remark}

 The paper is organized as follows. In Section 2, we recall some definitions and establish notation about persistence modules and barcodes. We briefly recall the definition and some facts about wrapped Floer homology in Section 3. In Section 4, we define barcode entropy and present two equivalent definitions for it. Section 5 contains the main ideas for proving Theorem \ref{thm:independence}, and its proof. Finally, in Section 6, we discuss the main ingredients that goes into the proof of Theorem \hyperref[theoremA]{A} and present its proof. 

 \subsection*{Acknowledgments} 
 \addtocontents{toc}{\protect\setcounter{tocdepth}{-1}}
 The author is deeply grateful to Viktor Ginzburg for his invaluable guidance during this project, to Başak Gürel for useful comments on the statement of Theorem A, and to Erol Burat for useful discussions. 
 \addtocontents{toc}{\protect\setcounter{tocdepth}{2}}
\section{Persistence modules and barcodes preliminaries}

In this section, we recall some basic definitions and facts about persistence modules and barcodes, focusing on the definitions that are more appropriate for our purposes. We point out that the definitions presented here are not the most general ones.  The reader may check \cite{polterovich2020topological} for a broader exposition and more details about persistence modules, and \cite{usher2016persistent} for a treatment of barcodes in the setting of Floer theory.

\subsection{Persistence modules} \label{sectionpersistencemodules} \label{persistence module}

 Fix a ground field $\mathbb{F}$. A persistence module is a pair $(V,\pi)$, where $V$ is a family of vector spaces $V_t$ over $\mathbb{F}$, with $t \in \mathbb{R}$, and $\pi$ is a family of linear maps $\pi_{st}$, from $V_s$ to $V_t$, for $s\leq t$ satisfying 
\begin{enumerate}[label=(\roman*)] \label{barcodes}
    \item (Persistence) $\pi_{ss} = Id$ and $\pi_{sr} = \pi_{tr} \circ \pi_{st}$, for any $s\leq t \leq r$,
    \item (Locally constant) There exists a subset $\mathcal{S} \subset \mathbb{R}$, called spectrum,  so that $\pi_{st}$ is an isomorphism for $s\leq t$ on the same connect components of $\mathbb{R} \setminus \mathcal{S}$,
    \item (q-tame) The maps $\pi_{st}$ have finite rank for any $s<t$,
    \item (Left-semicontinuous) For all $t\in \mathbb{R}$, 
    $$V_t = \underset{s<t}{\lim_{\longrightarrow}}V_s,$$
    \item (Lower bound) There exists $s_0 \in \mathbb{R}$ such that $V_s = 0$, for all $s \leq s_0$.
\end{enumerate}

We note that the definition of persistence module may vary in the literature. Nonetheless, in this paper, we refer to a pair $(V,\pi)$ that satisfies the above condition as a persistence module. 

Moreover, depending on the context, one may impose different conditions on the spectrum in order to have barcodes well behaved. For example, in the case of a non-degenerate Hamiltonian $H : M \times S^1 \rightarrow \mathbb{R}$  on a closed symplectic manifold $(M,\omega)$ satisfying enough conditions for the Hamiltonian Floer homology of $H$ to be well defined, it is easy to see that the persistence module given by the filtered Hamiltonian Floer homology of $H$, i.e., $V_t = HF^t(H)$ and persistence maps $HF^s(H) \rightarrow HF^t(H)$, for $s\leq t$, induced by the inclusion on the complex has spectrum a finite set. For our purposes, we need to consider the case where the spectrum $\mathcal{S}\subset \mathbb{R}$ is a closed, bounded from below nowhere dense set.

In what follows, it will be useful for us to reparameterize and truncate persistence modules. We establish the terminology through the following definition:

\begin{definition}
    For $f:\mathbb{R} \rightarrow \mathbb{R}$ an increasing function, $T\in \mathbb{R}$, and $(V,\pi)$ a persistence module, we define
    \begin{enumerate}[label={(\alph*)}]
        \item The \textbf{action reparametrization of $(V,\pi)$ by f}, denoted by $(V^f,\pi^f)$, to be the persistence module defined by 
        $$(V^f)_t = V_{f(t)}, \ \pi^f_{st} = \pi_{f(s)f(t)}.$$
        \item The \textbf{truncation of $(V,\pi)$ at $T$}, denoted by $tru((V,\pi),T) = (tru(V,T),tru(\pi,T))$ to be the persistence module defined by 
\begin{center}
    $tru(V,T) = 
    \begin{cases}
        V_t & \text{for } t < T, \\
        0 & \text{otherwise},
    \end{cases},$

    $tru(\pi,T) =
    \begin{cases}
        \pi_{st} & \text{for } j < T, \\
        0 & \text{otherwise}.
    \end{cases}$
\end{center}
    \end{enumerate}
\end{definition}

We will say that a persistence module $(\Tilde{V},\Tilde{\pi})$ is a \emph{reparametrization of} $(V,\pi)$, if $(\Tilde{V},\Tilde{\pi}) = (V^f,\pi^f)$ for some increasing function $f: \mathbb{R} \rightarrow \mathbb {R}$. One may find the notation $act(V,f)$ used for the action reparametrization of $(V,\pi)$ by $f$ in the literature. Whenever the maps $\pi_{st}$ are well understood in the context, we will simply write $V^f$ and $tru(V,T)$ for the action reparametrization and truncation, respectively.  

\subsection{Barcodes}

A persistence module $(V,\pi)$ as described above admits a unique (up to permutation) decomposition as a sum of interval persistence modules. More precisely, for an interval $I\subset \mathbb{R}$, we define the interval persistence module $(\mathbb{F}(I),\pi_{st}(I))$ by 
\begin{center}
    \begin{equation}
    \mathbb{F}(I)_{s} =  \begin{cases}
        \mathbb{F} \ \  \text{if} \ s \in I, \\
        0 \ \  \text{otherwise},
    \end{cases}, \ \pi_{st}(I) = \begin{cases}
        Id \ \ \text{if} \ s,t \in I, \\
        0 \ \  \text{otherwise}.
    \end{cases}
\end{equation}
\end{center}
\begin{theorem} (\text{Structure theorem}) To every persistence  module $(V,\pi)$ corresponds a multiset of intervals $B(V)$ such that 
$$(V,\pi) \cong \bigoplus_{I \in B(V)}(\mathbb{F}(I),\pi(I)) .$$
Moreover, this data is unique up to permutation.
\end{theorem}

The multiset  $B(V,\pi)$, or simply $B(V)$, if there is no risk for confusion, consisting of the intervals in this decomposition is referred to as the barcode of the persistence module. 

In what follows, it will be convenient for us to consider the number of not-too-short bars of a barcode. More precisely, for $\epsilon \geq 0$, we denote by $B_{\epsilon}(V,\pi)$, or simply $B_{\epsilon}(V)$, the bars in $B(V,\pi)$ with length greater or equal to $\epsilon$, and $b_{\epsilon}(V,\pi)$, or simply $b_{\epsilon}(V)$ the number of bars in $B_{\epsilon}(V,\pi)$ i.e., $b_{\epsilon}(V,\pi) = |B_{\epsilon}(V,\pi)|$. We call a persistence module $(V,\pi)$ \emph{moderate} if $b_{\epsilon}(V)$ is finite for all $\epsilon>0$. 

Another major theorem about persistence modules and barcodes that will be used on the main results is the so called isometry theorem. More precisely:
\begin{theorem} \label{theoremisopersistence} (\text{Isometry theorem}) For two persistence modules $V$ and $W$, we have
$$d_{int}(V,W) = d_{bot}(B(V),B(W).$$
where $d_{inter}$ and $d_{bottle}$ are the interleaving and bottleneck distance, respectively. 
\begin{remark} \label{remarktruboundedbytotal} First, note that the both terms in the equality above might be infinity depending on $V$ and $W$. Moreover, it is easy to see that for any $T>0$, then 
    $$d_{int}(tru(V,T),tru(W,T)) \leq d_{int}(V,W).$$
    Therefore from the Isometry theorem \ref{theoremisopersistence}, follows that
    $$d_{bot}(B(tru(V,T)),B(tru(W,T))) \leq d_{bot}(B(V),B(W)).$$
\end{remark}
\end{theorem}
\begin{remark} \label{remarkmoderatebarcodes}
    The results [\cite{chazal2016structure}, Theorem 5.21] and [\cite{le2021barcodes}, Proposition 22] imply that the space of moderate persistence modules is naturally isometric to a subspace to the completion of the space of finite barcode upper semi-continuous persistence modules bounded from below. 
\end{remark}

\begin{remark} \label{bepsilonversusthetruncation}
    We notice that truncating a persistence module can increase the number of its short bars, but not the number of not-too-short ones. Indeed, for a persistence module $(V,\pi)$ and $T \in \mathbb{R}$, and $\epsilon>0$, the number $b_{\epsilon}(V)$ of bar with length greater or equal to $\epsilon$ is greater or equal to the number $\text{\textcrb}_{\epsilon}(V,T-\epsilon)$ of bars in $B(V)$ with length greater or equal to $\epsilon$ beginning below $T - \epsilon$, and the later is equal to the number $b_{\epsilon}(tru(V,T))$ of bar in $tru(V,T)$ with length greater or equal to $\epsilon$.
\end{remark}
\section{Wrapped Floer homology} In this section, we discuss the basic construction of filtered wrapped Floer homology and recall some facts that will be useful for proving the main theorems of this paper.
\subsection{Exact conical Lagrangians} This subsection provides the definitions of the Lagrangians we consider in wrapped Floer homology, along with some remarks that will be useful. 

Let $(M,\lambda)$ be a Liouville domain, and $(L,\partial L) \subset (M,\partial M)$ a Lagrangian. We call $L$ \emph{asymptotically conical} if 
\begin{itemize}
    \item $\Lambda = \partial L$ is a Legendrian submanifold of $(\Sigma, \xi_{M})$, where $\Sigma = \partial M$, and $\xi_{M}$ is the contact structure induced by $\lambda|_{\Sigma} = \alpha$,
    \item $L \cap [1-\epsilon,1] \times \Sigma = [1-\epsilon,1] \times \Lambda$ for sufficiently small $\epsilon > 0$.
 \end{itemize}
We can extend $L$ to an Lagrangian $\widehat{L} \subset \widehat{M}$ by taking $\widehat{L} = L \cup_{\Lambda} [1,\infty)\times \Lambda$, where $\widehat{M}$ denotes the symplectic completion of $M$, i.e.,
$$\widehat{M} = M \cup_{\Sigma
} [1,\infty) \times \Sigma,$$
with the symplectic form $\omega=d\lambda$ extended to $[1,\infty)\times \Sigma$ as
$$\omega = d(r\alpha),$$
where $r$ is the coordinate in $[1,\infty)$.
We also call the extended Lagrangians in $\widehat{M}$ asymptotically conical. In what follows, we will deal with \emph{exact asymptotically conical} Lagrangians, i.e., Lagrangians $(L,\partial L) \subset (M,\Sigma)$ that are asymptotically conical, and $\lambda|_{L} = df$, for some $f \in C^{\infty}(L)$. We refer to Lagrangians satisfying the last condition as \emph{exact Lagrangians}. Throughout the rest of the paper, we assume that all the Legendrian and Lagrangian submanifolds considered are connected.
\begin{remark} \label{intersectionpointshaveaction0}
    For an exact asymptotically conical Lagrangian $L$, the primitive $f$ for $\lambda_{L}$ can be taken identically zero by a modification of $\lambda$ without changing the dynamics of the Reeb flow on the boundary $\partial M$. Indeed, since $f$ is constant on a neighborhood of the boundary, by adding a constant to it we can assume the constant is zero, i.e., $f\equiv 0$ on that neighborhood. Now we extend $f$ to a function $F$ on $M$ with compact support in M minus the collar neighborhood. By adding $-dF$ to $\lambda$, we obtain a new Liouville form, which we still denote by $\lambda$, so that $\lambda|_{L}=0$. This procedure changes $\lambda$ in the the interior of $M$ but not on the boundary. In the case of a pair of exact asymptotically conical Lagrangians $(L_{0},L_{1})$, we can not use the same argument to change $\lambda$ in such a way that $\lambda|_{L_{0}}=\lambda|_{L_{1}} \equiv 0$, but we can assume $f_{0}(x)=f_{1}(x)$, for all intersection points $x\in L_0 \cap L_1$, when $L_0 \pitchfork L_1$, by modifying $\lambda$ on a neighborhood of each intersection point. Notice that since $L_0 \pitchfork L_1$, there are no intersection on the collar and therefore this modification does not effect the dynamics on the boundary. 
\end{remark}   

\subsection{Hamiltonians and action filtration}
Now we present the basic construction of wrapped Floer homology and define the persistence module obtained from it. We mostly follow \cite{meiwes2018rabinowitz}, \cite{ritter2013topological} and \cite{cineli2023invariant}.

Let $(M,\lambda)$ be a Liouville domain, and fix $L_{0}$, $L_{1}$ exact asymptotically conical Lagrangians in $M$ with $\Lambda_{0}=\partial L_{0}$ and $\Lambda_{1}= \partial L_{1}$.  

A Hamiltonian $H: \widehat{M} \rightarrow \mathbb{R}$ is called \emph{admissible} if 
\begin{itemize}
    \item $H<0$ on $M$, and
    \item There exists $T>0$, $B \in \mathbb{R}$, and $r_0 \geq 1$ such that, $H(r,x)=rT - B$ on $[r_0,\infty)\times \partial M$.
\end{itemize}
We call $T$ the slope of $H$ at infinity and denote it by $slope(H)=T$. When H satisfies only the last condition we call it linear at infinity. In what follows, we will need to deal with Hamiltonians $H$ so that $H|_{M} \equiv 0$, and are linear at infinity. We will refer to these Hamiltonians as \textit{semi-admissible}. It will be convenient to consider the class of convex \emph{(semi-)admissible Hamiltonians}. These are (semi-)admissible Hamiltonians $H$ such that
\begin{itemize}
    \item $H|_M \equiv const$ ,
    \item $H(r,x) = h(r)$, in $[1,r_{\max}) \times \partial M$, for $h:[1,r_{\max}] \rightarrow \mathbb{R}$ with $h''>0$ in $(1,r_{\max})$,
    \item $H$ linear at infinity in $[r_{\max}, \infty) \times \partial M$, with $slope(H) = T$. 
\end{itemize}

We emphasize that $r_{\max}$ depends on $H$, i.e., $r_{\max} = r_{\max}(H)$, but for convenience we might omit $H$ when the Hamiltonian is well understood. Unless specified otherwise, all the Hamiltonians in this paper are considered to be convex (semi-)admissible (therefore not necessarily non-degenerate). We denote by $X_H$ the Hamiltonian vector field associated to a Hamiltonian $H:\widehat{M}\rightarrow \mathbb{R}$, defined by
$$\omega(X_H,\cdot) = -dH,$$
along with $\varphi_H^t$, for $t\in \mathbb{R}$, and $\varphi_H = \varphi_H^1$,  its Hamiltonian flow and Hamiltonian diffeomorphism respectively. The Reeb flow of $\alpha$ is denoted by $\varphi_{\alpha}^t$, with $t \in \mathbb{R}$.

We denote the set of smooth chords from $\widehat{L_0}$ to $\widehat{L_1}$ by $\mathcal{P}_{L_0 \rightarrow L_1}$, i.e.,
$$\mathcal{P}_{L_0 \rightarrow L_1} = \{\gamma:[0,T] \rightarrow \widehat{M} \ \mid \ \gamma(0) \in \widehat{L_0}, \ \gamma(T) \in \widehat{L_1}\}.$$
For a chord $\gamma : [0,T]\rightarrow \widehat{M} \in \mathcal{P}_{L_0 \rightarrow L_1}$, we call $T$ the \emph{length} and denote it by $length(\gamma)$. We denote by $\widehat{\mathcal{P}}_{L_0 \rightarrow L_1}$ the chords in $\mathcal{P}_{L_0 \rightarrow L_1}$ with length $1$. 

For a Hamiltonian $H$ in $\widehat{M}$, we consider $\mathcal{T}^T_{L_0\rightarrow L_1}(H)$ and $\mathcal{T}_{L_0\rightarrow L_1}(H)$, the set of Hamiltonian chords with length at most $T$, and the set of all Hamiltonian chords, from $\widehat{L_0}$ to $\widehat{L_1}$, respectively. Since the Hamiltonians considered here are time independent, then for $\gamma \in $$\mathcal{T}_{L_{0} \rightarrow L_{1}}(H)$, we have $\gamma(t) = \varphi_H^t(\gamma(0))$ on the interval where $\gamma$ is defined.

For the Reeb vector field $R_\alpha$ on the boundary $(\Sigma, \xi_{M})$ associated to the contact form $\alpha$, we denote by $\mathcal{T}^{T}_{\Sigma_{0} \rightarrow \Sigma_{1}}(\alpha)$, and $\mathcal{T}_{\Sigma_{0} \rightarrow \Sigma_{1}}(\alpha)$ the set of Reeb chords with length at most $T$, and the set of all Reeb chord, from $\widehat{\Lambda_0}$ and $\widehat{\Lambda_1}$, respectively. We denote by $\mathcal{S}(\alpha,\Lambda_{0}\rightarrow \Lambda_{1})$ the set of lengths of Reeb chords from $\Lambda_{0}$ to $\Lambda_{1}$. This is a closed nowhere dense subset of $[0,\infty)$. 

 Note that for a convex Hamiltonian on $[1,\infty) \times \Sigma$, 
$$X_H = h'(r)R_{\alpha}.$$
In this case, the Hamiltonian chords from $\widehat{L_0}$ to $\widehat{L_1}$ in $\{r\}\times \Sigma$, for $r \geq1$, correspond to Reeb chords of length $t= h'(r)$.

The action functional $\mathcal{A}^{L_{0} \rightarrow L_{1}}_{H} = \mathcal{A}_{H} : \widehat{\mathcal{P}}_{L_{0} \rightarrow L_{1}} \rightarrow \mathbb{R}$ associate to a Hamiltonian $H$ is defined by 
\begin{equation} \label{eq: action functional}
\mathcal{A}_{H}(\gamma) = f_{0}(\gamma(0)) - f_{1}(\gamma(1)) + \int^{1}_{0} \gamma^{*}\lambda - \int^{1}_{0}H(\gamma(t))dt,
\end{equation}
where for $i=0,1$, the $f_{i}$ is a smooth functions on $L_{i}$ such that $df_{i} = \lambda|_{L_{i}}$.  It is not hard to see that the critical points of $\mathcal{A}_{H}$ are the Hamiltonian chords of $H$ from $\widehat{L}_{0}$ to $\widehat{L}_{1}$ reaching $\hat{L}_{1}$ at time 1. We denote by $Crit(H,L_0 \rightarrow L_1)$ the set of critical points of $\mathcal{A}_{H}$. Notice that from Remark \ref{intersectionpointshaveaction0}, the action functional \eqref{eq: action functional} reduces to 
$$\mathcal{A}_{H}(\gamma) =  \int^{1}_{0} \gamma^{*}\lambda - \int^{1}_{0}H(\gamma(t))dt,$$
for chords $\gamma$ in $[1,\infty) \times \Sigma$.

For (semi-)admissible Hamiltonians, we can associate a reparametrization function that transforms lengths in actions. More precisely:
\begin{lemma}
    For a semi-admissible Hamiltonian $H$, with $H(r,x)= h(r)$ in $[1,r_{\max}(H))\times \Sigma$ and $H(r,x) = rT - B$, in $r \geq r_{\max}$, the reparametrization function $A_H =A_h$ that transforms lengths in actions given by $A_h(t)=rh^{'}(r) - h(r)$, where $h'(r)=t$ is a well defined,  continuous monotone increasing bijection between $[0,T]$ and $[0,B]$.
\end{lemma}
\begin{proof}
    For $1 \leq r < r' \leq r_{\max}$, then
    \begin{equation} \label{monotonicityofreparamentrization}
         rh'(r) -h(r)|^{r}_{r'} = \int^{r}_{r'} (rh'(r) - h(r))' dr = \int^{r}_{r'} rh''(r) dr,
    \end{equation}
    and by the convexity of $H$, 
    \begin{equation} \label{welldefinitinessofreparametrization}
       h'(r) - h'(r') = h'(r)|^{r}_{r'} \leq (rh'(r) -h(r))|^{r}_{r'} \leq r_{\max}h'(r)|^{r}_{r'} = r_{\max}(h'(r)- h(r')).
    \end{equation} 
    So if  $h'(r)=h'(r')$, then from \eqref{welldefinitinessofreparametrization}, we get $A_{h}(r)= A_{H}(r')$, which proves that $A_{h}$ is well defined. The monotonicity follows from \eqref{monotonicityofreparamentrization} since $h$ is convex and continuity is a consequence of the formula. Now, since 
    $$A_{h}(0) = h(0) = 0 \ \text{and} \ A_{h}(T) = r_{\max}h'(r_{\max}) - h(r_{\max}) = B,$$
    we conclude that $A_{h}$ is a bijection between $[0,T]$ and $[0,B]$. 
\end{proof}

There are three properties of the reparamentrization function that are going to be useful for us throughout the paper. We first note that from \eqref{welldefinitinessofreparametrization}, the reparametrization function is a bilipschitz, i.e.,
\begin{equation} \label{bilipschitz}
    t'-t \leq A_h(t') - A_h(t) \leq r_{\max}(H)(t'-t),
\end{equation}
for any $0 \leq t \leq t'\leq T$. Now for $H,K$ semi-admissible Hamiltonians with $H \leq K$, then $A_h \geq A_k$ in $[1,slope(H)]$, and $t \geq A_h(t)$, for all $t \in  [0,slope(H)]$. It is not direct to see this, but a simple way to justify it is as follows. Notice that for $t=h'(r)$, minus the action $-A_H(t)$ is the intersection of the tangent line to $h$ at $(r,h(r))$ with the ordinate. If $H \leq K$, and $h'(r_0) =t=k'(r_1)$, then the tangent lines to $h$ and $k$ at $(r_0,h(r_0))$ and $(r_1,k(r_1))$ respectively are parallel since they have the same slope, but the tangent line to $k$ is above the one to $h$ since $h \leq k$, and therefore $- A_k(t) \geq - A_h(t)$. Since $h$ is strictly convex in $(1,r_{\max}(H))$, the tangent line to $h$ at any point $(r,h(r))$ with $r > 1$ intersect the abscissa in $(1,\infty)$, so it is on the right to the line $r \rightarrow tr -t+1$, which is the line that passes through $1$ with slope $t$. The intersection of this line with the ordinate is $-t$, therefore $ -t \geq - A_h(t)$. It is not hard to see that for a semi-admissible Hamiltonian $H$, we have
\begin{equation} \label{A_H(t)convergetot}
    A_{sH}(t) \rightarrow t,
\end{equation}
as $s\rightarrow \infty$. See Figure \ref{functionA_H}.
\begin{center}
    \begin{tikzpicture} \label{functionA_H}
        \begin{axis}[xmin = -0.1, xmax = 3, axis lines = left, xlabel = $r$ , xtick = {1,2}, xticklabels = {$1$, $r_{\max}$}, xticklabel style={above left} ,ymin = -1.5, ymax = 3, ytick = \empty, axis lines = middle]
        \addplot[color = black, samples = 50, domain = 1:2 ]{(x-1)^2};
        \addplot[domain =2:3, color = black, samples = 50]{2*(x - 2) + 1}node[right,pos = 0.7]{$H$};
         \addplot[domain =-1:3, color = black, samples = 50]{(x - 3/2) + 1/4}; 
        \addplot[color = black, samples = 50, domain = 1:2 ]{2*(x-1)^2};
        \addplot[domain = 2:3, color = black, samples = 50]{4*(x - 2) + 2}node[right,pos = 0.1]{$K$};
        \addplot[domain =-1:3, color = black, samples = 50]{1*(x -5/4) + 1/8}; 
        \addplot[domain = -1:3, color = black]{x-1}; 
         \node[pin = -1: {$-A_{h}(t)$}] at (axis cs:0, -5/4) {};
         \node[pin = 3: {$-A_{k}(t)$}] at (axis cs: 0,-9/8) {};
         \node[pin = 80: {$-t$}] at (axis cs:0, -1) {};
        \end{axis}
        \node[below] at (current bounding box.south) {Figure \ref{functionA_H}};
    \end{tikzpicture}
\end{center}

We call a Hamiltonian $H:\widehat{M} \rightarrow \mathbb{R}$ non-degenerate for $L_{0}$ and $L_{1}$ if all elements in $Crit(H,L_0 \rightarrow L_1)$ are non-degenerate, i.e., 
$$\varphi_H(\widehat{L}_{0}) \pitchfork \widehat{L}_{1}.$$
Such Hamiltonians must have $slope(H) \notin \mathcal{S}(\alpha,\Lambda_{0}\rightarrow \Lambda_{1})$. 
\begin{remark}
    For a Hamiltonian $H$ in $\widehat{M}$ with $slope(H) \notin \mathcal{S}(\alpha,\Sigma_0 \rightarrow \Sigma_1)$, it is not necessary to make time-dependent Hamiltonian perturbations of $H$ to guarantee finitely many Hamiltonian chords and have $HW^*(H,L_0 \rightarrow L_1)$ well defined (unlike $SH^*(H)$, where time-dependent perturbation on $H$ are suitable; see \cite{salamon1999lectures}). Instead, it is enough make compact support Hamiltonian perturbations $\varphi(L)$, for $L=\widehat{L_0}$  or $L=\widehat{L_1}$ (or alternatively, keep the Lagrangians fixed and perturbing $H$ to $H \circ \varphi^{-1}$) to ensure non-degenerate Hamiltonian chords, i.e., $\varphi_H^1(\widehat{L_0}) \pitchfork \widehat{L_1}$. The advantage to perturb the Lagrangians (or the Hamiltonian) via a time-independent perturbation is that the Hamiltonian chords, on the case of a convex Hamiltonians, remain reparametrizations of Reeb chords; See \cite{ritter2013topological}.
\end{remark}

Notice that for a non-degenerate $H$ for $L_{0}$ and $L_{1}$, the elements in $Crit(H,L_0 \rightarrow L_1)$ have image contained in $M_{r_{\max}} = M \cup [1,r_{\max}]\times \Sigma$, where $r_{\max} = r_{\max}(H)$. In particular $\phi_{H}(\widehat{L}_0) \cap \widehat{L}_1$ is a finite set.

\subsection{Floer homology and continuation maps} In this subsection, we establish notation and outline the construction of wrapped Floer homology.

An almost complex structure $J$ on $([1,\infty) \times \Sigma, \lambda =r \alpha)$ is called \emph{cylindrical} if
\begin{itemize}
    \item Preserves $\xi_{M}$, i.e., $J(\xi_{M})= \xi_{M}$,
    \item $J|_{\xi_M}$ is independent of $r$,
    \item $J(r \partial_{r}) = R_{\alpha}$, for $r\geq 1$.
\end{itemize}

If these conditions hold only for $r \geq r_{0}$, for some $r_{0}>1$, then we call $J$ asymptotically cylindrical, or cylindrical at infinity. We say that an almost complex structure in $\widehat{M}$ is \emph{admissible} if it is compatible with $\omega$ and cylindrical at infinity. 

Fix an admissible almost complex structure in $\widehat{M}$, and consider the $\omega$-compatible Riemannian metric $g$ on $\widehat{M}$ satisfying $\omega(\cdot, J \cdot) = g(\cdot, \cdot)$. By endowing $\widehat{\mathcal{P}}_{L_{0}\rightarrow L_{1}}$ with the induced metric from $g$, the solutions of the gradient equation $\partial_{s}u =- \nabla \mathcal{A}_{H}(u)$, for a Hamiltonian $H$, are equivalent to strips $u : \mathbb{R} \times [0,1] \rightarrow \widehat{M}$, referred as Floer strips, satisfying
\begin{equation}
\label{Floer equation}
    \begin{cases}
    $$\partial_{s}u +J(u)(\partial_{t}u - X_{H}(u))=0$$, \\
    $$u(\cdot , 0) \in \widehat{L}_{0}, \ u(\cdot , 1) \in \widehat{L}_{1}$$.
    \end{cases}
\end{equation}

The moduli space of parameterized  Floer strips connecting the chords $x_-$ and $x_+$ with respect to $H$ and $J$ is defined by
$$ \mathcal{M}(x_-,x_+,H,J) = \{u:\mathbb{R} \times [0,1] \rightarrow \widehat{M} \mid u \ \text{satisfies} \ \eqref{Floer equation}, \ \lim_{s \rightarrow \pm \infty}u(s,\cdot)=x_{\pm} \}.$$
\begin{remark} \label{wrappedtolagrangian}
    We point out that for a non-degenerate Hamiltonian $H$, the moduli spaces above define the Lagrangian Floer homology $HF(\phi_H(L_0),L_1)$. Indeed, the Floer strips correspond to the pseudo-holomorphic strips with boundaries in  $\phi_{H}(\widehat{L_0})$ and $\widehat{L_1}$ solutions of $\partial_s v + \Tilde{J_t}\partial_t v = 0$ converging to the intersection points of $\phi_H(\widehat{L_0})$ and $\widehat{L_1}$ corresponding to the chords $x$ and $y$, where $\Tilde{J_t} = d \phi^{1-t}_H \circ J_t \circ d\phi^{t-1}_H$. For further  details about Lagrangian Floer homology, see, for example, \cite{ritter2009novikov}, and \cite{ritter2013topological}.
\end{remark}

The Floer equation is invariant by the $\mathbb{R}$-action $(s,u) \rightarrow u(s+\cdot, \cdot)$, so we consider the moduli space of unparameterized Floer strips $\widetilde{\mathcal{M}}(x,y,H,J)$ given by 
$$\widetilde{\mathcal{M}}(x,y,H,J)= \frac{{\mathcal{M}}(x,y,H,J)}{\mathbb{R}}.$$ 

The energy of a Floer strip 
$u$ is defined by
\begin{align} \label{energy}
    E(u) & = \int_{S^1 \times \mathbb{R}}|\partial_{s} u|^{2}_{g_{s}}ds\wedge dt, 
\end{align}
and one can see that for $u \in \mathcal{M}(x,y,H,J)$
\begin{equation} \label{descrease energy}
    E(u) = \mathcal{A}_H(x) - \mathcal{A}_H(y),
\end{equation}
which implies that $\mathcal{A}_H(x) \leq \mathcal{A}_H(y)$ when $\mathcal{M}(x,y,H,J)$ is nonempty. 

Consider a non-degenerate (semi-)admissible Hamiltonian $H$ and admissible almost complex structures $J$. For a generic perturbation of $J$ in $M_{r_{\max}}$, which we keep denoting by $J$, the moduli spaces $\mathcal{M}(x,y,H,J)$ are smooth manifolds for any $x,y \in \mathcal{T}_{L_0 \rightarrow L_1}(H)$. Furthermore, since $J$ is still cylindrical in $[1,\infty
)\times \Sigma$ after the perturbation, all curves in $\mathcal{M}(x,y,H,J)$ are contained in $M_{R}$, where $R=\max \{R(x),R(y),r_{\max}\}$; see \cite{ritter2013topological} for more details. For any action interval $I$, the filtered wrapped Floer homology of $(H,J)$ is the homology of the chain complex given by
$$CW^I(H,J, L_{0} \rightarrow L_{1} ) =  \bigoplus_{x \in Crit^I(H,L_0 \rightarrow L_1)} \mathbb{Z}_{2} \cdot x,$$
and differential 
$$\partial : CW \rightarrow CW, \ \partial x = \sum_{y \in Crit^I(H,L_0 \rightarrow L_1)}(\# \mathcal{M}_0(x,y,H,J) \mod 2) \cdot y,$$
where $\mathcal{M}_0(x,y,H,J)$ is the zero dimensional part of $\widetilde{\mathcal{M}}(x,y,H,J)$ (since the Hamiltonian $H$ is linear at infinity, and $J$ is cylindrical at infinity, the maximal principal applies and we have that $\mathcal{M}_0(x,y,H,J)$ is compact), and is denoted by $HW^I(H,J,L_{0}\rightarrow L_{1})$, or shortly $HW^I(H,J)$, when there is no confusion about the Lagrangians.  We write $HW^{a}(H,J)$ instead of $HW^{(- \infty, a)}(H,J)$, and $HW(H,J)$ instead of $HW^a(H,J)$, for any $a > slope(H)$.

Now we discuss the continuation maps. Let $(H_{-},J_{-})$ and $(H_{+},J_{+})$ be two pairs of non-degenerate (semi-)admissible Hamiltonians for $L_0$ and $L_1$ and admissible almost complex structures, with $H_{-}(x)\leq H_{+}(x)$ for all $x \in \widehat{M}$, which we denote by ($H_{-},J_-) \leq (H_{+},J_+)$, satisfying the necessary conditions so that $HW$ is well defined for them. Consider $(H_{s},J_{s})$ to be a homotopy between the two pairs such that $H_{s}=H_{\pm}$ and $J_{s}=J_{\pm}$ for $s$ close to $\pm \infty$, and $J_s$ admissible for all $s \in \mathbb{R}$. The Floer continuation strips for $(H_s,J_s)$ are the cylinders $u:\mathbb{R} \times [0,1] \rightarrow \widehat{M}$ satisfying 

\begin{equation} 
\label{continuationstrips}
    \begin{cases}
        $$\partial_{s}u +J_{s}(u)(\partial_{t}u - X_{H_{s}}(u))=0,$$ \\
        $$u(\cdot , 0) \in \widehat{L}_{0}, \ u(\cdot , 1) \in \widehat{L}_{1}.$$
    \end{cases}
\end{equation}

For $x_{-}\in Crit(H_-)$ and $x_{+}\in Crit(H_+)$, the moduli space of Floer continuation strips connecting the chords $x_{-}$ and $x_{+}$ with respect to $(H_{s},J_{s})$ is defined by
$$\mathcal{M}(x_{-},x_{+},H_{s},J_{s}) = \{u:\mathbb{R} \times [0,1] \rightarrow \widehat{M} \mid u \ \text{satisfies} \ \eqref{continuationstrips}, \ \lim_{s \rightarrow \pm \infty}u(s,\cdot)=x_{\pm}\}.$$

For generic time dependent perturbation of the pair $(H_s,J_s)$ in $M_{r_{\max}}$ ($r_{\max}$ is such that $r_{\max} \geq r_{\max}(H_s)$, for all $s\in \mathbb{R}$), which we keep denoting by $(H_s,J_s)$, the moduli spaces $\mathcal{M}(x_-,x_+,H_s,J_s)$ are smooth manifold with compactification on broken trajectories for any $x_- \in \mathcal{T}_{L_0 \rightarrow L_1}(H_-)$ and $x_+ \in \mathcal{T}_{L_0 \rightarrow L_1}(H_+)$ \cite[for more details]{ritter2009novikov,salamon1999lectures}. We denote by $\mathcal{M}_{0}(x_{-},x_{+},H_{s},J_{s})$ the zero dimensional part of $\mathcal{M}(x_{-},x_{+},H_{s},J_{s})$. Due to the term $J_{s}$ on equation \eqref{continuationstrips}, $\mathbb{R}$-reparametrizations of solutions of \eqref{continuationstrips} are no longer solutions. For non-decreasing homotopies $(H_{s},J_{s})$ where $J_{s}$ are admissible for $r\geq r_{\max}$, for all $s \in \mathbb{R}$, the maximum principle holds, so all curves in $\mathcal{M}(x_{-},x_{+},H_{s},J_{s})$ are contained in $M_{R}$, where $R=\max \{R(x_{\pm}),r_{\max}\}$ \cite[for more details]{ritter2013topological}.

One can show that the energy of a Floer continuation strips (similarly defined as the energy of Floer strips) $u \in \mathcal{M}(x_-,x_+,H_sJ_s)$ satisfies
\begin{align} \label{energycontinuation}
    E(u)  = \mathcal{A}_{H_{-}(x_{-})} - \mathcal{A}_{H_{+}}(x_{+}) - \int \partial_{s} H_{s}(u)ds \wedge dt.
\end{align}

For monotone homotopies $(H_{s},J_{s})$, and an interval $I$, the continuation map is defined by (notice that from \eqref{energycontinuation}, the differential does not increase action)

$$\mathcal{\hat{X}}_{H_{-}\rightarrow H_{+}}^I: CW^I(H_{-},J_{-}) \rightarrow CW^I(H_{+},J_{+}),$$
with 
$$\mathcal{\hat{X}}^I_{H_{-}\rightarrow H_{+}}(x_{-}) = \sum_{x_{+} \in Crit^I(H_+)} (\# \mathcal{M}^{H_{s}}_{0}(x_{-},x_{+},H_s,J_s) \mod2 )\cdot x_{+},$$
and by an standard argument it can be shown that $\mathcal{\hat{X}}^I_{H_{-}\rightarrow H_{+}}$ descents to a map $\mathcal{X}^I_{H_{-}\rightarrow H_{+}}: HW^I(H_{-},J_{-}) \rightarrow WH^I(H_{+},J_{+})$ in homology (see for example \cite{salamon1999lectures}).

\begin{lemma} 
    \label{conpro}
    \begin{enumerate}
        \item Changing the monotone homotopy $(H_{s},J_{s})$ changes $\mathcal{\hat{X}}^I_{H_{-}\rightarrow H_{+}}$ by a chain homotopy, so $\mathcal{X}^I_{H_{-}\rightarrow H_{+}}$ is independent of $(H_{s},J_{s})$. 
        \item For any $H \leq K \leq G$, then $\mathcal{X}^I_{K \rightarrow G} \circ \mathcal{X}^I_{H \rightarrow K} = \mathcal{X}^I_{H \rightarrow G}$.
        \label{property3}
        \item The constant homotopy $(H_{s},J_{s}) = (H,J)$ for all $s \in \mathbb{R}$ induces the identity in $HW^I(H,J)$.
        \item For $H_{\pm}$ with the same slope at infinity, then $\mathcal{X}_{H_{-}\rightarrow H_{+}}$ is an isomorphism.
        \label{property4}
    \end{enumerate}
\end{lemma}
For the proof of Lemma \ref{conpro}, see for example \cite{ritter2013topological} and \cite{ritter2009novikov}. From \hyperref[property3]{\emph{(3)}} in the lemma above, it can be shown that the wrapped Floer homology of $(H,J)$ is independent of $J$, so we write $HW^a(H)$, instead of $HW^a(H,J)$, for any $a\in \mathbb{R}$, and from part \hyperref[property4]{\emph{(4)}} in Lemma \ref{conpro}, follows that the wrapped Floer homology for a Hamiltonian $H$ depends only on it's slope at infinity, i.e., $HW(H)=HW(H^{T})$, where $T= slope(H)$.

Now let $H$ any Hamiltonian (not necessarily non-degenerate) with $slope(H) \notin \mathcal{S}(\alpha,\Lambda_0 \rightarrow \Lambda_1)$ and an action interval $I\subset \mathbb{R}$. Then the filtered wrapped Floer homology $HW^I(H)$ is readily defined over $\mathbb{Z}_2$ as long as the end points of $I$ are not in $\mathcal{S}(H)$. This is the wrapped homology $HW^I(\Tilde{H})$ of a small non-degenerate perturbation of $H$ with $slope(\Tilde{H})=slope(H)$. Using the continuation map, it is easy to see that $HW^I(\Tilde{H})$ is independent of $\Tilde{H}$ when $\Tilde{H}$ is sufficiently close to $H$. We write $HW^a(H):=HW^{(-\infty,a]}(H)$ for simplicity. 

For $a \leq b$ with $a,b \notin S(H)$, we have the well defined map "inclusion" map 
$$i_H^{a,b}: HW^a(H)\rightarrow HW^b(H).$$

Now we extend the definition of $HW^a(H)$ for $a \in \mathbb{R}$ by considering
$$HW^a(H) := \lim_{\begin{smallmatrix} \longrightarrow & \\ a' \leq a \end{smallmatrix}}HW^{t'}(H),$$
where we assume $t' \notin \mathcal{S}(H)$. The "inclusion" maps naturally extend to these homology spaces. We notice that $i^{a,b}_H$ are isomorphisms as long as the interval $[a,b)$ has not point in $\mathcal{S}(H)$. Also, note that
$$HW^a(H) = 0, \ \text{for} \ a\leq 0,$$
for semi-admissible Hamiltonians. Therefore, we have that $a \rightarrow HW^a(H)$ together with the "inclusion" maps define a persistence module as in defined in Section \ref{persistence module}.
\begin{remark} \label{perturbation}
    The wrapped Floer homology is insensitive  to perturbations of the Hamiltonian $H$ and $a \in \mathbb{R}$, when $slope(H) \notin \mathcal{S}(\alpha,L_0 \rightarrow L_1)$ and $a \notin \mathcal{S}(H)$. More precisely, for two Hamiltonians $H$ and $H'$ $C^0$-close to each other outside of the domain where they are linear at infinity, and $a$ close to $a'$ meeting the conditions above,  there is a natural isomorphism 
    $$HW^a(H) \cong HW^{a'}(H').$$
\end{remark}

Using Remark \ref{perturbation}, it is easy to see that continuation maps between two Hamiltonians naturally extends to the degenerate case through small perturbation of the Hamiltonians.

Now we work towards removing the assumption that $T=slope(H) \notin \mathcal{S}(\alpha,L_0 \rightarrow L_1)$. To this end, for any $a\in \mathbb{R}$ we define 
\begin{equation} \label{defwrapdeg}
    HW^a(H) := \lim_{\begin{smallmatrix}
    \longrightarrow \\ H' \leq H
\end{smallmatrix}}HW^a(H'),
\end{equation}
where the limit is taken over the Hamiltonians $H'\leq H$ with $slope(H') \notin \mathcal{S}(\alpha,L_0 \rightarrow L_1)$. Clearly in this case, $slope(H') < slope(H)$. The relevant case for us is that of the total wrapped Floer homology, i.e., $HW^a(H)$, for $a>\sup \mathcal{S}(H)$ for a semi-admissible Hamiltonian $H$ with $slope(H) \in \mathcal{S}(\alpha,L_0 \rightarrow L_1)$. When $H$ is semi-admissible, we notice that Definition \ref{defwrapdeg} is equivalent to
\begin{equation} \label{deferapdegs}
    HW(H):= \lim_{\begin{smallmatrix}
    \longrightarrow \\ 0<s'<1
\end{smallmatrix}}HW(sH),
\end{equation}
where the limit is taken over $0 < s'<1$ with $s'T\notin \mathcal{S}(\alpha,L_0 \rightarrow L_1)$, since the Hamiltonians involved in the Definition \ref{deferapdegs}
form a cofinal sequence for the Hamiltonians involved in Definition \ref{defwrapdeg}. This equivalent definition for semi-admissible Hamiltonians is more suitable for our purposes. From Definition \ref{deferapdegs}, we have that for $0\leq s$, the wrapped Floer homology groups for a semi-admissible Hamiltonian $H$ can be obtained by as
$$HW(sH) = \lim_{\begin{smallmatrix}
    \longrightarrow \\ s'\leq s
\end{smallmatrix}}HW(s'H),$$
 where the limit is taken over $s'\leq s$ with $s'T \notin \mathcal{S}(\alpha,L_0 \rightarrow L_1)$. 
By setting $HW(sH) = 0$, if $s<0$, we have that $s\rightarrow HW(sH)$ is defined for all $s\in \mathbb{R}$ regardless of $sT \notin \mathcal{\alpha}$, and define a persistence module with structural maps the "inclusion" maps.

\subsection{Filtered wrapped Floer homology}
In this subsection, we focus on reviewing the definitions and properties of filtered wrapped Floer homology of semi-admissible Hamiltonians. The approach used here mostly follows the ideas in \cite{cineli2024barcode} and \cite{cineli2023invariant} and references therein. Most of the theorems contained in this subsection are proved in one of the two papers for the case of symplectic homology. However, we present all the proofs with minor modifications for the case of wrapped Floer homology for the sake of completeness.

\
The filtered wrapped Floer homology $HW^a(M,L_0 \rightarrow L_1)$, for $a>0$ is defined as
\begin{equation} \label{defwrap}
    HW^a(M,L_0 \rightarrow L_1) := \lim_{\begin{smallmatrix}
    \longrightarrow \\ H
\end{smallmatrix}} HW^a(H,L_0 \rightarrow L_1),
\end{equation}
where the limit is taken over the set of admissible Hamiltonians. Since convex admissible Hamiltonian form a cofinal sequence, we can consider $H$ on this class. Furthermore, we set
\begin{equation} \label{conventionwrapa<0}
    HW^a(M,L_0 \rightarrow L_1) := 0 \ \text{when} \ a \leq 0.
\end{equation}

\begin{theorem} \label{teohomolyofHtohomologyofH}
    Let $H$ be a semi-admissible Hamiltonian with $T = slope(H)$. Then for every $a\leq T$, there exists an isomorphism
    \begin{equation} \label{isowraptohamilton}
        \Phi^a_H : HW^a(M,L_0 \rightarrow L_1) \rightarrow HW^{A_H(a)}(H,L_0 \rightarrow L_1).
    \end{equation}
    Moreover, these isomorphisms are natural in the sense that they commute with the "inclusion" and monotone continuation maps. More precisely, for any $a<a'\leq T$, and two semi-admissible Hamiltonians $H'\leq H$, the diagrams
\begin{center} 
        \begin{tikzcd}
        HW^{a'}(M,L_0 \rightarrow L_1) \arrow[d] \arrow[r, "\Phi_H^{a'}"] & HW^{A_H(a')}(H,L_0 \rightarrow L_1) \arrow[d] \\
        HW^{a}(M,L_0 \rightarrow L_1) \arrow[r, "\Phi_H^a"] & HW^{A_H(a)}(H,L_0 \rightarrow L_1), 
    \end{tikzcd}
    \end{center}
    and
     \begin{center} 
        \begin{tikzcd}
        HW^a(M,L_0 \rightarrow L_1) \arrow[d, "id"] \arrow[r, "\Phi_{H'}^a"] & HW^{A_{H'}(a)}(H',L_0 \rightarrow L_1) \arrow[d] \\
        HW^a(M,L_0 \rightarrow L_1) \arrow[r, "\Phi_H^a"] & HW^{A_H(a)}(H,L_0 \rightarrow L_1), 
    \end{tikzcd}
    \end{center}
    commute, where the vertical arrows are the "inclusion" maps in the first diagram and the right vertical arrow is the monotone continuation map in the second.
    
\end{theorem}
A consequence of Theorem \ref{teohomolyofHtohomologyofH}, central for this paper, is the following corollary. It states that the filtered wrapped Floer homology defines a persistence module, in the sense of Section \ref{sectionpersistencemodules}, which is not evident from the definition. Moreover, we will also use Theorem \ref{teohomolyofHtohomologyofH} for computing the wrapped Floer homology barcode entropy of $(M,\lambda,L_0 \rightarrow L_1)$ using semi-admissible Hamiltonians.
\begin{corollary}
    The family of vector spaces $a \rightarrow HW^a(M,L_0 \rightarrow L_1)$ is a persistence module, with structure maps the direct limit of the "inclusion" maps. 
\end{corollary}
In order to prove Theorem \ref{teohomolyofHtohomologyofH}, we will we will make use of the following two lemmas:
\begin{lemma} \label{lemmawrapusingHamiltonians}
    Let $H$ be any semi-admissible Hamiltonian. Then we have 
    \begin{equation} \label{lemmadef}
        HW^a(M,L_0 \rightarrow L_1) = \lim_{\begin{smallmatrix}
        \longrightarrow \\ s \rightarrow \infty
    \end{smallmatrix}}HW^a(sH),
    \end{equation}
    for any $a \in \mathbb{R}$.
\end{lemma}
\begin{lemma} \label{lemmainvariance}
    Let $H_s$, with $s\in [0,1]$, be a homotopy of Hamiltonians such that all Hamiltonians $H_s$ have the same slope, and let $I_s$ of intervals continuously depending on $s$ such that for all $s$, the end points of $I_s$ are outside $\mathcal{S}(H_s)$. Then there is a natural isomorphism
    $$HW^{I_0}(H_0) \xrightarrow{\cong} HW^{I_1}(H_1).$$
\end{lemma}
\begin{proof}[Proof of Lemma \ref{lemmawrapusingHamiltonians}]
    The lemma follows basically from Remark \ref{perturbation} and the definitions. When $a<0$, then it follows directly from the \eqref{conventionwrapa<0}. Now we assume $a>0$.
    Clearly, in \eqref{lemmadef}, we can replace the direct limit as $s\rightarrow \infty$ by the direct limit over any monotone increasing sequence $s_i \rightarrow \infty$, and the limit is independent of the sequence. Now, fix any such a sequence and any monotone non-decreasing sequence $\epsilon_i \rightarrow 0^+$. Then $H_i = s_iH - \epsilon_i$ forms is a cofinal sequence, which we can use to obtain \eqref{defwrap}. On the other hand,
    $$HW^a(H_i) = HW^{a-\epsilon}(s_iH).$$
    
    Assume first that $a\notin S(\alpha,L_0 \rightarrow L_1)$. Then $a-\epsilon \notin \mathcal{S}(\alpha,L_0 \rightarrow L_1)$ for all large $i$, $a-\epsilon$ and $a$ are in the same connected component of the complement to $\mathcal{S}(H)$ (this is the point where it is essential that $a\neq 0$). Therefore, from Remark \ref{perturbation}, we get
    $$HW^{a-\epsilon}(s_iH) = HW^a(s_iH),$$
    and the statement follows by passing to the limit as $i \rightarrow \infty$.
    Now, assume that $a \in \mathcal{S}(\alpha,L_0 \rightarrow L_1)$. Consider a sequence $s_i \rightarrow \infty$ so that $a \notin S(s_iH)$. then 
    $$HW^a(s_iH) = HW^t(s_iH),$$
    for all $t\in [a-\delta_i,a+\delta_i]$, for some $\delta_i$ depending on $i$. Now choose a sequence $\epsilon_i \rightarrow 0^+$ with $\epsilon_i < \delta_i$. Then, by Remark \ref{perturbation}
    $$HW^a(s_iH)= HW^{a-\epsilon_i}(s_iH) = HW^a(H_i),$$    
    for all large $i$. Passing to the limit as $i \rightarrow \infty$, we obtain \eqref{lemmadef}.
    \end{proof}

\begin{proof}[Proof of Lemma \ref{lemmainvariance}]
    First, notice that it is enough to prove that lemma assuming that either the intervals or the Hamiltonian is independent of $s$. Then, we partition the interval $[0,1]$ into sufficiently small intervals obtaining a concatenation of homotopies, and we conclude the result. For a fixed Hamiltonian $H$ and varying intervals, the result follows from taking a sufficiently small non-degenerate perturbation $\Tilde{H}$ for which Remark \ref{perturbation} holds for the end points of the intervals $I_s$. For a fixed interval and vary Hamiltonian, assuming that all Hamiltonians $H_s$ are $C^{\infty}$ close to each other, we take a perturbation $\Tilde{H}$ that is suitable for all the Hamiltonians on the interval $I_s = I$, or more precisely, $HW^I(\Tilde{H}) = HW^I(H_s)$, for all $s$.
\end{proof}
Let $H_0 \leq H_1$ be two semi-admissible Hamiltonians. We will assume that $r_{\max}(H_0) = r_{\max}(H_0) = r_{\max}$ (this condition could be replaced by $r_{\max}(H_1) \leq r_{\max}(H_0)$ with careful wording modification). We can consider
$$f = f_{H_0,H_1} = A_{H_1} \circ A_{H_0}^{-1} : [0,A_{H_0}(r_{\max})] \rightarrow [0,A_{H_1}(r_{\max})].$$
Notice that $f$ is a monotone increasing function since $A_{H_0}$ and $A_{H_1}$ are (it is a monotone increasing bijection if $slope(H_0) =slope(H_1)$) and $f(t) \leq t$, for all $t$, since $A_{H_1} \leq A_{H_0}$ for $H_0 \leq H_1$. Furthermore, it is not hard to see that 
\begin{equation} \label{specH_0toH_1}
    f(\mathcal{S}(H_0)) = f([0,A_{H_0}(r_{\max})]) \cap \mathcal{S}(H_1),
\end{equation}
i.e., $f$ gives rise to a one-to-one correspondence between the action spectra as long as target in the range of $f$.
    \begin{proposition} \label{propositiotoprovetheorem}
        For all $a < A_{H_0}(r_{\max})$, not in $\mathcal{S}(H_0)$, there are isomorphisms of the Floer homology groups
        \begin{equation} \label{isoH_0H_1}
            HW^a(H_0) \xrightarrow{\cong} HW^{f(a)}(H_1).
        \end{equation}
        These isomorphisms are natural in the sense that they commute with "inclusion" maps and monotone homotopies.
    \end{proposition}
    \begin{proof}
Let $r_0 \in [1,r_{\max})$ be uniquely determined by $A_{H_0}(r_0) = a$. Consider an intermediate semi-admissible Hamiltonian $H_{01}$ with $r_{\max}(H_{01})=r_{\max} = r_{\max}(H_0)=r_{\max}(H_1)$ such that
\begin{itemize}
    \item $H_0 \leq H_{01} \leq H_1$,
    \item $H_01 = H_0$ in $[1,r_0],$
    \item $slope(H_{01}) = slope(H_1)$.
\end{itemize}
Note that with this choice $a \notin \mathcal{S}(H_{01})$. 

Then a monotone increasing homotopy from $H_0$ to $H_1$ induces a isomorphism 
\begin{equation} \label{isoH_0H_01}
    HW^a(H_0) \xrightarrow{\cong} HW^a(H_{01}).
\end{equation}
This follows since by the maximal principal since all Floer cylinders from $H_0$ to $H_{01}$ and Floer continuation trajectories starting at $1$-periodic orbits with action less than or equal $a$ lie in the region $M_{r_0}$.

Now consider $H_s$ be a linear monotone increasing homotopy from $H_{01}$ to $H_1$. Denote by $f_s = A_{H_s} \circ A_{H_0}^{-1}$ the resulting family of maps, and by \eqref{specH_0toH_1}, note that $f_s(a) \notin \mathcal{S}(H_s)$ since $a \notin \mathcal{S}(H_0)$. Applying the Lemma \ref{lemmainvariance} to the family of intervals $I_s=(-\infty,f_s(a)]$ we obtain an isomorphism 
\begin{equation} \label{isoH_01H_1}
    HW(H_{01}) \xrightarrow{\cong} HW^{f_1(a)}(H_1) = HW^{f(a)}(H_1),
\end{equation}
where in the last identity we have used the fact that $f = f_{H_{01},H_1}= f_1$ on $[0,a]$ since $H_{01} = H_1$ on $[0,r_0]$. The isomorphism \eqref{isoH_01H_1} now follows from composing the isomorphisms \eqref{isoH_0H_01} and \eqref{isoH_01H_1}. The assertion that the isomorphisms are natural follow from the construction of the isomorphisms \eqref{isoH_0H_01} and \eqref{isoH_01H_1}.
\end{proof}

\begin{proof}[Proof of theorem \ref{teohomolyofHtohomologyofH}] Consider the family of Hamiltonians $s \rightarrow sH$ and let $f_s = f_{H,H_s}$. From Proposition \ref{propositiotoprovetheorem} we have isomorphisms
$$HW^a(H) \xrightarrow{\cong} HW^{f_s(a)}(sH).$$

Recall, $A_{sH}(t) \rightarrow t$ as $s \rightarrow \infty$ for all $t$ (see \eqref{A_H(t)convergetot}), and we conclude that $f_{s}(a) \rightarrow A_{H_0}^{-1}(a)$ as $s \rightarrow \infty$. Then passing to the limit as $s \rightarrow \infty$ and applying Lemma \ref{lemmawrapusingHamiltonians}, we obtain the inverse of the isomorphism \eqref{isowraptohamilton}. From the constructions in Lemma \ref{lemmawrapusingHamiltonians} and Proposition \ref{propositiotoprovetheorem}, the isomorphisms are natural, i.e., the diagrams in \ref{teohomolyofHtohomologyofH} commute.
\end{proof}
We denote the barcode and the number of not-too-short bars for the persistence module \( a \rightarrow HW^a(H, L_0 \rightarrow L_1) \) for some (semi-)admissible Hamiltonian by \( B(H, L_0 \rightarrow L_1) \) and \( b_{\epsilon}(H, L_0 \rightarrow L_1) \), respectively. Similarly, for the persistence module \( a \rightarrow HW^a(M, L_0 \rightarrow L_1) \), we denote the barcode and the number of not-too-short bars by \( B(M, L_0 \rightarrow L_1) \) and \( b_{\epsilon}(M, L_0 \rightarrow L_1) \), respectively. One should expect similar simplifications on the notation for the other operations with persistence modules and barcodes in the setting of wrapped Floer homology. When the Lagrangians are well understood, we may suppress it from the notation. 

Assume $\varphi:\widehat{M} \rightarrow \widehat{M}$ is a compactly supported Hamiltonian diffeomorphism, i.e., $\varphi = \varphi_H$ for a compactly supported Hamiltonian $H: \mathcal{S}^1 \times \widehat{M} \rightarrow \mathbb{R}$. We define the Hofer norm of $\varphi$ by
$$ ||\varphi|| = \inf_{H \in \mathcal{D}_{\varphi}} \int_{0}^1 (\sup_{x\in \widehat{M}} H(t,x) - \inf_{x \in \widehat{M}} H(t,x)) dt,$$ 
where the infimum is taken over the set $\mathcal{D}_{\varphi}$ of compactly supported Hamiltonians $H:S^1 \times \widehat{M} \rightarrow \mathbb{R}$ with $\varphi = \varphi_H$. 

For a pair of Hamiltonian isotopic Lagrangians $\widehat{L_0}$ and $\widehat{L_1}$ in $\widehat{M}$ that admit a compactly supported Hamiltonian isotopy (in this case all the Lagrangians along the isotopy coincide on the cylindrical end), we define the Hofer distance between them to be
$$d_H(L_0,L_1) = \inf\{||\varphi_{H}|| \mid \varphi_H(\widehat{L_0}) = \widehat{L_1}\}.$$

We notice that the barcode is reasonably insensitive to small perturbation of the Lagrangians on the wrapped Floer homology with respect to the Hofer norm. More precisely, for two exact asymptotically conical Lagrangians $L_0$ and $L_1$ in $M$ with $L_0\pitchfork L_1$, and a Hamiltonian $H$, we have the following: For any compactly supported Hamiltonian perturbation $L'$ of $\widehat{L_1}$ with $d_H(\widehat{L_1},L')< \frac{\delta}{2}$, and $\widehat{L_0} \pitchfork L'$ in $M$, we have (we abuse the notation and keep denoting $L' \cap M$ by $L'$)
\begin{equation} \label{barcodeofpertubation}
    b_{\epsilon + \delta}(H, L_0 \rightarrow L_1) \leq b_{\epsilon}(H,L_0 \rightarrow L') \leq b_{\epsilon - \delta}(H,L_0 \rightarrow L_1).
\end{equation}
The proof follows from Remark \ref{wrappedtolagrangian}, and the result for the Lagrangian Floer homology case. We refer to \cite{cineli2021topological,mailhot2022spectral,polterovich2020topological,usher2016persistent}, and references therein.
\begin{remark} \label{intersectionoboundingnottooshortbars} 
    For a pair of exact asymptotically conical Lagrangians $L_0$ and $L_1$ in $M$, and a (semi-)admissible non-degenerate Hamiltonian $H$, we notice that
    \begin{align*}
      |\phi_{H}(L_0) \cap L_1| = |\mathrm{dim}(CW(H,L_0 \rightarrow L_1))| = 2b(H,L_0 \rightarrow L_1) - \mathrm{dim}(HW(H,L_0 \rightarrow L_1)),
    \end{align*}
    and therefore,
    $$|\phi_{H}(L_0) \cap L_1| \geq b(H,L_0 \rightarrow L_1) \geq b_{\epsilon}(H,L_0 \rightarrow L_1),$$
    for any $\epsilon>0$.
\end{remark}
\begin{proposition} \label{propostiondistances}
For $H,K: \widehat{M} \rightarrow \mathbb{R}$ non-degenerate Hamiltonians with $r_{\max}(H),r_{\max}(K)\leq r_{\max}$, then
      \begin{equation} \label{botboundedbyabovebycinfitynorm}
          d_{bot}(B(H),B(K)) \leq||H - K||_{\infty},r_{\max}.
      \end{equation} 
\end{proposition}
\begin{proof}
    We first notice that for a non-degenerate Hamiltonian $H$, the chain complexes $CW(H)$ and $CW(H+c)$ have the same generators since $X_{H} = X_{H+c}$, and the actions of a generator $x \in Crit(H,L_0 \rightarrow L_1)$ are related by $\mathcal{A}_{H+c}(x) = \mathcal{A}_{H}(x)-c$. Moreover, the Floer equations \eqref{Floer equation} for both Hamiltonian coincide, so we conclude that 
    \begin{equation} \label{translationHamiltonian}
        HW^a(H+c) = HW^{a+c}(H) = HW^a(H)[c].
    \end{equation}
    From the definition of wrapped Floer homology for a Hamiltonian $H$, we conclude that \eqref{translationHamiltonian} still holds.   
   
   Now let $\delta = ||H-K||_{\infty,r_{\max}}$. Then
    $$H \leq K + \delta, \ \text{and} \  K \leq H+\delta \ \text{in} \ M_{r_{\max}},$$
    and by the continuation maps and the first item, we get the maps
    $$\Phi= \mathcal{X}^{H \rightarrow K + \delta}_a : HW^a(H) \rightarrow HW^a(K)[\delta],$$
    and
    $$\Psi = \mathcal{X}^{K \rightarrow H + \delta}_a : HW^a(K) \rightarrow HW^a(H)[\delta],$$
    such that the diagrams
    \begin{center}
        \begin{tikzcd}[row sep=large, column sep=large]
            HW^a(H) \arrow[r, "\Phi"]  \arrow[rr, "i^{a, a+ \delta}_H" ,bend right=20 ] & HW^a(K)[\delta] \arrow[r, "{\Psi[\delta]}"] & HW^a(H)[2\delta],
        \end{tikzcd}
    \end{center}
    and
    \begin{center}
        \begin{tikzcd}[row sep=large, column sep=large]
            HW^a(K) \arrow[r, "\Psi"]  \arrow[rr, "i^{a,a + \delta}_K" ,bend right=20 ] & HW^a(H)[\delta] \arrow[r, "{\Phi[\delta]}"] & HW^a(K)[2\delta],
        \end{tikzcd}
    \end{center}
    commute for all $a \in \mathbb{R}$, which proves that 
    $$d_{int}(HW^a(H),HW^a(K)) \leq||H - K||_{\infty,r_{\max}}.$$
    From the isometry Theorem \ref{theoremisopersistence}, we conclude \eqref{botboundedbyabovebycinfitynorm}. 
\end{proof}
\begin{corollary}
    For any semi-admissible Hamiltonian $H$ in $\widehat{M}$, the persistence module $V_s = HW^s(H)$ is moderate, i.e., $b_{\epsilon}(H)$ is finite for all $\epsilon>0$.
\end{corollary}
\begin{proof}
   Since for any non-degenerate Hamiltonian $\Tilde{H}$, the persistence module $V_s =HW^s(\Tilde{H})$  is finite barcode upper semi-continuous persistence modules bounded from below, then from propositions \ref{propostiondistances} and \ref{remarkmoderatebarcodes}, for a semi-admissible Hamiltonian $H$, the persistence module $V_s=HW^s(H)$ is moderate. 
\end{proof} 
\section{Barcode entropy} In this section, we recall the definition of wrapped Floer homology \emph{barcode entropy}, and present two equivalent definitions for it. 
\subsection{Wrapped Floer homology barcode entropy} Here we recall the definition of the main character of this paper. We point out that definition presented below is slightly different from the one in the introduction, but we will discuss why they are equivalent.

For a Liouville domain $(M,\lambda)$ and two exact asymptotically conical Lagrangians $L_0$ and $L_1$, we can define the wrapped Floer homology barcode entropy in the following way:
\begin{definition}
    The wrapped Floer homology \textbf{barcode entropy} of $M$ is defined by 
    $$\hbar^{HW}(M,L_0 \rightarrow L_1) =\lim_{\epsilon \rightarrow 0} \limsup_{t \rightarrow \infty} \frac{\log^{+}b_{\epsilon}(tru(M,L_0 \rightarrow L_1,t))}{t},$$
\end{definition}
The definition above is in the same fashion as the definition of symplectic homology barcode entropy, first introduced in \cite{fender2023barcode}. Denoting by $\text{\textcrb}_{\epsilon}(M,L_0 \rightarrow L_1,t)$ the number of bar with length greater or equal to $\epsilon$ and left end point at most $t$ in the barcode $B(M,L_0 \rightarrow L_1)$, then from Remark \ref{bepsilonversusthetruncation}, we have 
$$\text{\textcrb}_{\epsilon}(M,L_0 \rightarrow L_1,t-\epsilon) \leq b_{\epsilon}(tru(M,L_0 \rightarrow L_1,t))\leq \text{\textcrb}_{\epsilon}(M,L_0 \rightarrow L_1,t),$$
and we conclude that the wrapped Floer homology barcode entropy could have been defined using $\text{\textcrb}_{\epsilon}(M, L_0 \rightarrow L_1,t)$ instead of the $b_{\epsilon}(tru(M, L_0 \rightarrow L_1,t))$ (as we did in the introduction for simplicity), aligning with the definition in \cite{ginzburg2022barcode}. The reason why we prefer the definition incorporating the truncation will be clear on the proof of theorem A. 

Again, we will drop the Lagrangians from the notation when they are well understood. 
\subsection{Equivalent definitions for barcode entropy}

In this subsection we discuss two equivalent definitions for the wrapped Floer homology barcode entropy. The second one will be useful on the proof of theorem \hyperref[theoremA]{A}. Both equivalent definitions are presented here on the spirit of ideas in 
\cite{fender2023barcode}.

In what follows, by $(M,\lambda, L_0 \rightarrow L_1)$, we mean a Liouville domain $(M,\lambda)$ with two exact asymptotically conical Lagrangians $L_0$ and $L_1$ satisfying $L_0 \pitchfork L_1$. Consider a semi-admissible Hamiltonian $H$ in $(M,\lambda, L_0 \rightarrow L_1)$, and let $\epsilon >0$ smaller than $\mathcal{S}_{min}(H) = A_H(\mathcal{S}_{min})$, where $\mathcal{S}_{min}$ is the smallest length of a Reeb chord from $\Lambda_0$ to $\Lambda_1$, i.e., $\mathcal{S}_{min} = \inf\mathcal{S}_{\alpha}(\Lambda_0 \rightarrow \Lambda_1)$. For any $t>\epsilon$, the exact sequence 

$$0\rightarrow CW^{<\epsilon}(H,J) \xrightarrow{i} CW^{<t}(H,J) \xrightarrow{\pi} CW^{[\epsilon,t)}(H,J)\rightarrow 0, $$
induces the exact triangle
\begin{center}
\begin{tikzcd}[column sep=small]
& HW^{\epsilon}(H) \arrow[dr,"i_*"] & \\
  HW^{[\epsilon,t)}(H)  \arrow[ur,"\delta"]   &                         & \arrow[ll,"\pi_*"] HW^{<t}(H),
\end{tikzcd}  
\end{center}
and by the choice of $H$ and $\epsilon$, it easy to see that

\begin{center}
\begin{tikzcd}[column sep=small]
& HL(L_0,L_1) \arrow[dr,"i_*"] & \\
  HW^{[\mathcal{S}_{min}(H),t)}(H)  \arrow[ur,"\delta"]   &                         & \arrow[ll,"\pi_*"] HW^{<t}(H).
\end{tikzcd}  
\end{center}

For any other semi-admissible Hamiltonian $K$ with $H \leq K$, the continuation maps induce a map from the long exact sequence for $H$ to the long exact sequence for $K$. Since by \eqref{A_H(t)convergetot} $\mathcal{S}_{min}(sH) \rightarrow \mathcal{S}_{min}$ as $s\rightarrow \infty$, it is easy to see that by taking the limit on each term, we obtain the exact triangle

\begin{center}
\begin{tikzcd}[column sep=small]
& HL(L_0,L_1) \arrow[dr,"i_*"] & \\
  HW^{[\mathcal{S}_{min},t)}(M)  \arrow[ur,"\delta"]   &                         & \arrow[ll,"\pi_*"] HW^{<t}(M).
\end{tikzcd}  
\end{center}

We define $HW_+^a(M,L_0 \rightarrow L_1) := HW^{[\mathcal{S}_{min},a)}(M,L_0 \rightarrow L_1)$ and consider the persistence module $V_a= HW^a_+(M,L_0 \rightarrow L_1)$, for $a\geq \mathcal{S}_{min}(H)$, and $V_a = 0$ otherwise, with structure maps the "inclusion" maps. The barcode of this persistence module is represented by $B_+(M,L_0 \rightarrow L_1)$, and its barcode entropy, to be refereed as the positive wrapped Floer homology barcode entropy of $(M,\lambda,L_0 \rightarrow L_1)$ is denoted by $\hbar^{HW^+}(H,L_0 \rightarrow L_1)$.

Notice that the family $W_a =  HF(L_0,L_1)$ for $a \geq 0$, and $W_a = 0$ otherwise, is a persistence module where the structure maps $W_a\rightarrow W_{a'}$, for $a\leq a'$, are trivial if $a<0$ and $Id$ otherwise.   

\begin{proposition} \label{proppositiventropyvsentropy}
The positive wrapped Floer homology barcode entropy of $(M,\lambda,L_0 \rightarrow L_1)$ coincide with the wrapped Floer homology barcode entropy. i.e.,
    \begin{equation} \label{entropyvspositiveentropy}
        \hbar^{HW^+}(M,L_0 \rightarrow L_1) = \hbar^{HW}(M,L_0 \rightarrow L_1).
    \end{equation}
\end{proposition}

For the proof of Proposition \ref{proppositiventropyvsentropy}, we will make use of the following theorem, which the proof can be found in \cite{buhovsky2022coarse}.
\begin{theorem}[\cite{buhovsky2022coarse}, Theorem 3.1] \label{theoremexactsequencebarcode}
    Let $U \rightarrow V \rightarrow W$ be an exact sequence of moderate persistence modules. Then for every $\epsilon >0$, the following inequality holds: 
    $$b_{2\epsilon}(B(V)) \leq b_{\epsilon}(B(U)) + b_{\epsilon}(B(W)).$$
\end{theorem}

\begin{proof}[Proof of Proposition \ref{proppositiventropyvsentropy}]
Consider the exact sequences 
$$tru(HL(L_0,L_1),t) \xrightarrow{tru(i)} tru(HW(M),t) \xrightarrow{tru(\pi)} tru(HW^+(M),t),$$
and
$$tru(HW(M),t) \xrightarrow{tru(i)} tru(HW_+(M),t) \xrightarrow{tru(\pi)} tru(HL(L_0,L_1),t).$$
By denoting $c=\mathrm{dim}(HL(L_0,L_1))$ for simplicity, from Theorem \ref{theoremexactsequencebarcode} we get
$$b_{2\epsilon}(tru(B(M),t)) \leq b_{\epsilon}(tru(B_+(M),t)) + c,$$
and
$$b_{2\epsilon}(tru(B_{+}(H),t)) \leq b_{\epsilon}(tru(B(H),t)) + c.$$
Therefore
$$b_{2\epsilon}(tru(B(M),t)) - c \leq b_{\epsilon}(tru(B_+(M),t)) \leq b_{\epsilon/2}(tru(B(M),t)) + c,$$
and
$$b_{2\epsilon}(tru(B_+(M),t)) - c \leq b_{\epsilon}(tru(B(M),t)) \leq b_{\epsilon/2}(tru(B_+(M),t)) + c.$$
By taking the exponential growth with respect to $t$, we obtain  
$$\hbar_{2\epsilon}^{HW^+}(H,L_0\rightarrow L_1) \leq \hbar_{\epsilon}^{HW}(H,L_0\rightarrow L_1) \leq \hbar_{\epsilon/2}^{HW^+}(H,L_0 \rightarrow L_1).$$
Finally, taking $\epsilon \rightarrow 0$, we conclude \eqref{entropyvspositiveentropy}.
\end{proof}

Now we consider a Liouville domain $(M,\lambda)$, and two exact asymptotically conical Lagrangians $L_0$ and $L_1$. For any semi-admissible Hamiltonian $H$ in $\widehat{M}$ with $H(r,x)=rT-B$, for $r \geq r_{\max}(H)$, and any $s>0$, we can consider the barcode associated to the persistence module $t \rightarrow HW^t(sH)$. We use this family of barcodes to define the \emph{wrapped Floer homology barcode entropy of $H$} as follows.

\begin{definition}
For a semi-admissible Hamiltonian $H$, we define the wrapped Floer homology barcode entropy of $H$ to be
$$\hbar^{HW}(H,L_0 \rightarrow L_1) =\lim_{\epsilon \rightarrow 0} \limsup_{s \rightarrow \infty} \frac{\log^{+}b_{\epsilon}(tru(sH,sB))}{sT}.$$
\end{definition}

\begin{proposition} \label{sequecialbarcodeentropy}
    For a semi-admissible Hamiltonian $H$ in a Liouville domain $(M,\lambda)$, and two exact asymptotically conical Lagrangians $L_0$ and $L_1$,  we have
    $$ \hbar ^{HM}(H,L_0 \rightarrow L_1) = \hbar ^{HM}(M,L_0 \rightarrow L_1).$$
\end{proposition}

The strategy of the proof is as follows. We start proving that for any semi-admissible Hamiltonian $H$, with $H(r,x) = rT -B$ for $r \geq r_{\max}$, and any $\epsilon > 0$, we have
\begin{equation} \label{inequalitytoprovetheorem}
    b_{r_{\max}\epsilon}(tru(H,B))\leq b_{\epsilon}(tru(M,T)) \leq b_{\epsilon}(tru(H,B)).
\end{equation}
Then, for the family of semi-admissible Hamiltonians $s\rightarrow sH$ we get
$$b_{r_{\max}\epsilon}(tru(sH,sB))\leq b_{\epsilon}(tru(M,sT))\leq b_{\epsilon}(tru(sH,sB)),$$
and therefore, we can conclude that the exponential growth of the barcodes coincide. 
\begin{proof}[Proof of Proposition \ref{sequecialbarcodeentropy}]
Let $H$ be any semi-admissible Hamiltonian with $H(r,x) = rT -B$, for $r\geq r_{\max}(H)$. Then from Theorem \ref{teohomolyofHtohomologyofH}, the persistence modules $tru(HW^a(M,L_0 \rightarrow L_1),T)$ and $tru(HW^{A_H(a)}(H,L_0 \rightarrow L_1),T) = tru(HW^a(H,L_0 \rightarrow L_1),T)^{A_H}$ are isomorphic (the truncation is essential here since Theorem \ref{teohomolyofHtohomologyofH} gives no information regarding the maps when $a > T$). Now since the function $A_H$ is bilipschitz, i.e.,
$$t'-t \leq A_h(t') - A_h(t) \leq r_{\max}(t'-t),$$
for any $0 \leq t \leq t'\leq T$, we conclude that \eqref{inequalitytoprovetheorem} holds for any $\epsilon>0$.

\end{proof}
\section{Independence of the fillings}

In this section, we discuss how the wrapped Floer homology barcode entropy of a Liouville domain depends only on the dynamics on the contact boundary; in other words, the barcode entropy is independent of the filling.\\

Throughout this section we fix $(\Sigma,\lambda,\Lambda_0 \rightarrow \Lambda_1)$, where $(\Sigma,\lambda)$ is a compact contact manifold, and $\Lambda_0$ and $\Lambda_1$ are Legendrian submanifolds. We say that $(M,\lambda,L_0 \rightarrow L_1)$ is a filling of $(\Sigma,\alpha,\Lambda_0 \rightarrow \Lambda_1)$, where $(M,\lambda)$ is a Liouville domain, and $L_0$ and $L_1$ are exact asymptotically conical Lagrangians, if $\partial M = \Sigma$, $\lambda|_{\Sigma} = \alpha$, and $\partial L_0 =\Lambda_0 $ and $\partial L_1=\Lambda_1$, with $L_0 \pitchfork L_1$. Consider $(M,\lambda,L_0 \rightarrow L_1)$ a filling of $(\Sigma, \alpha,\Lambda_0 \rightarrow \Lambda_1)$, and fix a collar neighborhood $Z$ of $\Sigma$ in $M$ such that 
$(Z,\lambda|_Z) \simeq (\Sigma \times [\frac{1}{2},1], r\alpha)$, where $r$ denotes the coordinate in $[\frac{1}{2},1]$. This way we see $\Sigma \times \{\frac{1}{2}\}$ as a submanifold of $M$. The next lemma is a well known fact about pseudo-holomorphic curves. We refer the reader to \cite{salamon1999lectures} and \cite{mcduff2012j} for more details.

\begin{lemma} \label{lemmaenergyofcurvesboundedbybelow}
    For any admissible almost complex structure $J$ in $\widehat{M}$,  there exists a constant $C=C(J)$, such that if $u$ is a $J$-holomorphic curve having boundary components in $\Sigma \times\{\frac{1}{2}\}$, $\Sigma \times \{1\}$, $\widehat{L}_0$ and $\widehat{L}_1$ then $E(u)\geq C$.
\end{lemma}
We point out that for a (semi-)admissible Hamiltonian $H$, if a Floer strip $u$ has $E(u)< C$ and connects two critical points of $H$ that are entirely contained in $\Sigma \times [1,\infty)$, then $u$ does not cross $\Sigma \times \{\frac{1}{2}\}$. Indeed, since in this case $H|_{\Sigma\times [\frac{1}{2},1]}$ is constant, then the Floer equation \eqref{Floer equation} reduces to classical $J$-holomorphic equation, and the energy of the Floer strip reduces to the energy of the $J$-holomorphic curve.

The next lemma is helpful to compare the barcodes of two different fillings. In what follows we denote by $B^C_+$ the bars in a barcode $B$ that have left end point positive and length at most C.

\begin{lemma}[\cite{fender2023barcode}, Lemma 4.8]  \label{lemmacomparingbarcode} 
    Let $(V_1,l_1)$, $(V_2,l_2)$ and $(W,l_W)$ be orthogonalizable $\mathbb{K}$-spaces. Consider $(C = V_1 \oplus W,l_1 \oplus l_W)$ (resp $(D = V_2 \oplus W,l_2 \oplus l_W)$) endowed with a linear map $\partial_C : C \rightarrow C$ (resp $\partial_D :D \rightarrow D$) such that $\partial_C^2 = 0$ (resp $\partial_D^2=0$) and $\partial_C$ decreases action, where 
    $$(l_i \oplus l_W)(v,w) = \max\{l_i(v),l_W(w)\}, \ \text{for} \  i=1,2.$$
    Assume that there exists constants $\eta>0$ and $E>0$ such that
    \begin{enumerate}
        \item $l_i(v_i) < E < l_W(w), \ \text{for all} \ v_i \in V_i, \ w \in W, \ \text{for} \ i=1,2,$
        \item $l_W(\pi_W(\partial_C w ) - \pi_W(\partial_Dw)) < l_W(w) - \eta, \ \text{for all} \ w\in W.$ ($W$ is viewed as a subspace of $V_i \oplus W$, and $\pi_W$ means the natural projection to $W$ from $V_i \oplus W$.)
    \end{enumerate}
    Denote by $\mathbb{B}_C = \{x_1^C, ... , x_m^C, y_1^C,...,y^C_{n_1},z^C_1,...,z^C_{n_1}\}$ the orthogonal basis of $C$ such that $\partial_C x_i^C =0$, $\partial_C z_j^C = y
        _j^C$. Let $C_E^{<\eta}$ be the multi-set of intervals $(l_C(y_j^C), l_C(z_j^C)]$ such that $l_C(z_j^C) - l_C(y_j^C) < \eta$, and $E<l_C(y_j^C)$. Define $D_E^{<\eta}$ similarly. Then $C_E^{<\eta} = D_E^{<\eta}$.
\end{lemma}

\subsection{Proof of Theorem \ref{thm:independence}} For two fillings $(M,\lambda,L_0\rightarrow L_1)$ and $(N,\eta,K_0 \rightarrow K_1)$ of $(\Sigma, \alpha, \Lambda_0 \rightarrow \Lambda_1)$, and a (semi-)admissible Hamiltonian $H$ in $\Sigma \times [\frac{1}{2},\infty)$ such that $H|\equiv d$ on the collar neighborhood of $\Sigma$, we can naturally extend $H$ to $\widehat{M}$ and $\widehat{N}$ by setting $H|_M \equiv d$ and $H|_N \equiv d$. We keep denoting the extensions by $H$. Assuming that $H$ in non-degenerate (notice that $H$ is non-degenerate for $\widehat{M}$ if and only if it is non-degenerate for $\widehat{N}$ since the Lagrangian fillings are assumed to be transverse), and considering admissible almost complex structures $J^M$ and $J^N$ for $\widehat{M}$ and $\widehat{N}$ respectively such that $J^M \equiv J^N = J$ in $\Sigma \times [\frac{1}{2},\infty)$, we can associate the filtered chain complexes $CW(H,L_0 \rightarrow L_1)$ and $CW(H,K_0 \rightarrow K_1)$. By considering
$$W = \bigoplus_{x\in Crit^{(0,\infty)}(H)} x \cdot \mathbb{Z}_2,$$
we can write
\begin{equation} \label{CW(M)}
    CW(H,L_0 \rightarrow L_1) = V_M \oplus W,
\end{equation}
and 
\begin{equation} \label{CW(N)}
    CW(H,K_0 \rightarrow K_1) = V_N \oplus W,
\end{equation}
where $V_M = \bigoplus_{x\in L_0 \cap L_1} x \cdot \mathbb{Z}_2$ and $V_N = \bigoplus_{x\in K_0 \cap K_1} x \cdot \mathbb{Z}_2$. We notice that from Lemma \ref{lemmaenergyofcurvesboundedbybelow}, the conditions in Lemma \ref{lemmacomparingbarcode} hold for \eqref{CW(M)} and \eqref{CW(N)} with $0<E<S_{min}(H)$ and $\eta = C(J) = C$, so we conclude that $B_+^{<C}(H,L_0 \rightarrow L_1) = B^{<C}_+(H,K_0 \rightarrow K_1)$.

In what follows we fix a filling $(M,\lambda,L_0 \rightarrow L_1)$ of $(\Sigma,\alpha,\Lambda_0 \rightarrow \Lambda_1)$. By Remark \ref{bepsilonversusthetruncation}, we know that for any (semi-)admissible non-degenerate Hamiltonian $H$ in $\widehat{M}$ with $slope(H) = T$,  the bars in $B_{\epsilon}(tru(H,L_0 \rightarrow L_1,t+\epsilon))$, for any $\epsilon,t>0$, are equivalent to bars in $B(H,L_0 \rightarrow L_1)$ satisfying 
\begin{itemize}
    \item left end point is at most $t$,
    \item Length at least $\epsilon$.
\end{itemize}
We classify the bars in $B(H,L_0 \rightarrow L_1)$ in the following types:
\begin{itemize}
    \item[Type] I: Bars longer than $\epsilon$ with right end point greater than $t$,
    \item[Type] II: Bars longer than $\epsilon$ with left end point smaller than $t$,
    \item[Type] III: Bars shorter than $\epsilon$ with right end point greater than $t$,
    \item[Type] IV: Bars shorter than $\epsilon$ with right end point smaller than $t$.
\end{itemize}
Denoting by $b^{\epsilon}(X,t)$ the number of bars in $B(tru(H,L_0 \rightarrow L_1,t + \epsilon))$ of type $X$, we obtain
$$b^{\epsilon}(I,t) + 2b^{\epsilon}(II,t) + 2b^{\epsilon}(III,t) +b^{\epsilon}(IV,t) = |L_0 \cap L_1| + |Crit^{(0,t]}(H,L_0 \rightarrow L_1)|.$$
Since $b^{\epsilon}(III,T)$ and $b^{\epsilon}(IV,T)$ are independent of the filling for $\epsilon<C$, we conclude
$$n_{\epsilon}(t) =  b^{\epsilon}(I,t) + 2b^{\epsilon}(II,t) - |L_0 \cap L_1| = |Crit^{(0,t]}(H,L_0 \rightarrow L_1)| - 2b^{\epsilon}(III,t) - b^{\epsilon}(IV,t),$$
is independent of the filling. Furthermore, since 
$$b_{\epsilon}(tru(H,L_0 \rightarrow L_1, t)) = b^{\epsilon}(I,t) + b^{\epsilon}(II,t),$$
\begin{equation} \label{equationindependence}
    \frac{n_{\epsilon}(t) + |L_0 \cap L_1|}{2}\leq b_{\epsilon}(tru(H,L_0 \rightarrow L_1,t + \epsilon)) \leq n_{\epsilon}(t) + |L_0 \cap L_1|.
\end{equation}
Now consider a semi-admissible Hamiltonian $H$ with $H(r,x) =rT - B$ for $r\geq r_{\max}(H)$, and a sequence $s_n \rightarrow \infty$ such that $s_nT \notin \mathcal{S}(\alpha,L_0 \rightarrow L_1)$. Let $\widetilde{H}_n$ a semi-admissible non-degenerate perturbation of $s_nH$ (notice that it is not necessary to perturb $s_nH$ inside of the filling since the Lagrangian fillings intersect transverse) such that $r_{\max}(\widetilde{H}) = r_{\max}(H) =r_{\max}$ and $\widetilde{H} = H$ in $r \geq r_{\max}$, with $||\widetilde{H} - H||_{\infty,r_{\max}}<\frac{\epsilon}{2}$. Then by Proposition \ref{propostiondistances} and Remark \ref{remarktruboundedbytotal}, we can conclude that 
\begin{equation} \label{equationnottooshortbarsperturbation}
    b_{\epsilon}(B(tru(\widetilde{H}_n,s_nB)))\leq b_{2\epsilon}(B(tru(s_nH,s_nB)))\leq b_{4\epsilon}(B(tru(\widetilde{H}_n,s_nB))).
\end{equation}
From \eqref{equationnottooshortbarsperturbation} and \eqref{equationindependence}, we get
\begin{equation} \label{equationsanduiche}
    \frac{n_{\epsilon}(s_nB - \epsilon) + |L_0 \cap L_1|}{2} \leq b_{2\epsilon}(B(tru(s_nH,s_nB))) \leq n_{4\epsilon}(s_nB - 4\epsilon) + |L_0 \cap L_1|.
\end{equation}
Therefore, for $4\epsilon<C$, by taking the exponential growth in \ref{equationsanduiche} we obtain that $\hbar^{HW}_{2\epsilon}(H,L_0 \rightarrow L_1)$ is independent of the filling. Finally by taking $\epsilon \rightarrow 0^+$ and Proposition \ref{sequecialbarcodeentropy}, conclude \eqref{equationindependence}.
\hfill\qed

\section{Interplay between entropies}

In this section, we discuss the main ingredients essential for proving Theorem \hyperref[theoremA]{A}, along with its proof. Although the ideas of Lagrangian tomograph and Crofton's inequality are well known by experts, we introduce them here highlighting the necessary adaptations for their application within our specific framework. An essential step for the proof of Theorem \hyperref[theoremA]{A} is based on Yomdin's Theorem \cite{yomdin1987volume}, which implies, in particular, that the exponential growth of the volume of a submanifold under the iterates of a smooth map on a manifold is bounded from above by the topological entropy of the map. 
\subsection{Lagrangian tomograph}
In this subsection we discuss the idea of Lagrangian tomograph. We mostly follow \cite{bae2022comparison} for constructing a Lagrangian tomograph of a Lagrangian $L$ keeping its structure at infinity. Other types of Lagrangian tomograph may be found, for example in \cite{cineli2021topological,fender2023barcode}, and references therein.

Let $L \subset M$ be a exact asymptotically conical Lagrangian submanifold, with $\Lambda = \partial L \subset  \Sigma$, and a number $r_0 \geq 1$. We want to construct a family $(L_s)_{s\in B^d}$ of Lagrangian submanifolds of $\widehat{M}$ in a tubular neighborhood of $\widehat{L}$, smoothly parameterized by a ball $B_{r} = B_{r}(0) \subset \mathbb{R}^d$, for some large $d$, satisfying the following properties:

\begin{enumerate}
    \item Each $L_s$ in the family is exact and conical in $[r_0,\infty) \times \Sigma$,

    \item The Lagrangians $L_s$ are Hamiltonian isotopic to $\widehat{L}$ via some compactly supported Hamiltonian, and $d_{H}(L_s,L)$ can be taken as small as desired, i.e., $d_H(L_s,\widehat{L}) < \frac{\delta}{2}$, for all $s \in B_r$, for any chosen $\delta$,
    \item For any Lagrangian $K$ in $\widehat{M}$, the Lagrangians $L_s$ are transverse to $K$ for almost all $s$.
\end{enumerate}
Such family will be called an \emph{exact asymptotically conical Lagrangian tomograph} for $L$ conical in $[r_{0},\infty)\times \Sigma$. In what follows, we construct such a Lagrangian tomograph for an exact asymptotically conical Lagrangian $L\subset M$.

We first remark that from the Weinstein Lagrangian tubular neighborhood, it is enough to construct such family in a neighborhood of the zero section $\widehat{L}$ of $T^* \widehat{L}$.

Let $\{\psi_{i}: \widehat{L} \rightarrow \mathbb{R}\}^d_{i=1}$ be a family of smooth functions on $\widehat{L}$ such that $\{(d\psi_{i})_{x}\}^d_{i=1}$ span $T_{x} \widehat{L}$, for all $x \in \widehat{L}$, and consider $\mu: \widehat{L} \rightarrow \mathbb{R}$ an smooth function with $\mu|_{\widehat{L} \cap M_{r_{0}}} > 0$ (we abuse notation and think of $M_{r_0}$ as $M \cup [1,r_0) \times \Sigma$) and $\mu(r,x)|_{[r_{0},\infty) \times L}=0$, such that for $\eta_i = \mu \psi_i$, $\{(d\eta_{i})_{x}\}^d_{i=1}$ still spans $T_{x} \widehat{L}$ for all $x \in \widehat{L} \cap M_{r_{0}}$. Then consider the family 
$$f_{s}(x) = \sum^{d}_{i=1} s_{i} \eta_i(x),$$
with $s=(s_{1},...,s_{d}) \in \mathbb{R}$. For $r>0$ small enough, restricting the parameter $s$ to a ball $B^d = B_{r}(0) \subset \mathbb{R}^d$, the map $\Psi: B^d \times \widehat{L} \rightarrow T^{*} \widehat{L}$, with $\Psi(s,x) = d f_{s}(x)$ is a submersion on its image and satisfies the desired properties. 
\begin{remark}
    We point out that the family of Lagrangians $(L_s)_{s\in B}$ in the construction can be taken to be $C^{\infty}$ close to $\widehat{L}$ as long as we take the radius of $B=B_r(0)$ small enough.
\end{remark}
\begin{remark}
    Another important remark is that for any sequence of Lagrangians $(L_n)_{n \geq 1}$ in $\widehat{M}$, by choosing $\delta$ small enough, the Lagrangians $L_s$, conical in $[r_0,\infty)\times \Sigma$, are transverse to $L_n$ for all $n$ and almost all $s\in B$. Moreover, the functions
    $$N_n(s) = |L_s \cap L_n \cap M_{r_0}|,$$
    are well defined for all $n$, i.e., $N_n(s)$ is finite for almost all $s$ and integrable with respect to the Lebesgue measure $ds$ in $B$.
\end{remark}
\subsection{Crofton's inequality} Crofton's inequalities can be found in the literature in more generality than the version we present here. The usage of this technique goes back to Arnold in \cite{arnold1990dynamics} and \cite{arnol1990dynamics}. Even though the idea of the Crofton's inequality is well-known by experts, we present a proof of the version that is suitable for our purposes for the sake of completeness. For more details, see \cite{cineli2021topological} and references therein.

We begin by introducing a general set up where the inequality holds. Consider a compact manifold $B$ possible with boundary, a compact (also possible with boundary) submanifold $L \subset B$, and $ds$ a smooth measure in $B$. We denote by $\pi: E = B \times L \rightarrow B$ the projection on the first coordinate. Here, $B$ serves as the set of parameters, and we denote the points on it by $s$. We also consider a Riemannian manifold $M$, not necessarily compact, and a submersion,
$$\Psi : E \rightarrow M,$$
onto its image. Even though $M$ is not compact, the image $\Psi(E)$ is, since $E = B \times L$ is compact. Lets assume that $\Psi_{s} \coloneqq \Psi|_{\{s\}\times L}$ to be an embedding for all $s \in B$, and we denote $L_{s} = \Psi_{s}(L)$ the image of $L$ by $\Psi_{s}$ in $M$. Lastly, we consider $L' \subset M$ a closed submanifold of $M$ with 
$$\mathrm{codim} \ L' = \mathrm{dim} \ L.$$
Since $\Psi$ is a submersion, then in particular $\Psi$ is transverse to $L'$, therefore $\Psi_{s} \pitchfork L'$ for almost all $s \in B$ (after a $C^1$ perturbation of $\Psi$, we can assume $\Psi_{s} \pitchfork L'$ for all $s \in B$), and consider
$$N(s) \coloneqq |L_{s} \cap L'|.$$
\begin{lemma}[Crofton's inequality] In the above setting,
\begin{equation}\label{crofton}
    \int_{B}N(s)ds \leq const \cdot \mathrm{vol}(L'),
\end{equation} 
where the constant depends only on the measure $ds$, $\Psi$ and the metric in $M$, but not in $L'$.
\end{lemma}
\begin{proof}
    Set $N = \Psi^{-1}(L')$. $N$ is a submanifold of $E$, and
    $$\mathrm{codim} \ N = \mathrm{codim} \ L' = \mathrm{dim} \ L,$$
    therefore, $\mathrm{dim} \ N = \mathrm{dim} \ B$. By definition,
    \begin{align} \label{equaN(s)}
        N(s) = |L_s \cap L| = |(\{s\}\times L) \cap N|.
    \end{align}
    Lets fix metrics on $B$ and $L$, and endow $E$ with the product metric. Denote by $ds$ the volume density associated to the product metric. Then from \eqref{equaN(s)}, we get
    $$\int_B N(s)ds = \int_N \pi^*ds \leq \text{vol}(N).$$
    where $\pi^*ds$ is the pull-back measure or pull-back density by restricted to $N \subset B\times L$. The last inequality follows from the fact that $D\pi$ is a orthogonal projection on each fiber, so it can only decrease the volume. Therefore, for an arbitrary metric on $E$, and an arbitrary metric $ds$ on $B$, we have
    \begin{equation} \label{vollower}
        \int_{B}N(s)ds \leq const \cdot \text{vol}(N),
    \end{equation}
    for some constant.

    Now we equip $E$ with a Riemannian metric such the restriction of $D\Psi$ to the normals of the fibers of $\Psi$, i.e., $\Psi^{-1}(y), \ y \in M$ are isometries. Then by Fubini's theorem, 
    \begin{align*}
        \text{vol}(N) = & \int_{L'}\text{vol}(\Psi^{-1}(y))dy|_{L'} \\
    & \leq \max_{y \in \Psi(E)} \text{vol}(\Psi^{-1}(E))\text{vol}(L'),
    \end{align*}
    where $dy|_{L'}$ stands for the restriction measure to $L'$. Thus, for any arbitrary metric on $E$, 
    \begin{equation} \label{volupper}
        \text{vol}(N) \leq const \cdot \text{vol}(L').
    \end{equation}
    Therefore, by combining the inequalities \eqref{vollower} and \eqref{volupper}, we obtain \eqref{crofton}.
    
\end{proof}
\subsection{Proof of Theorem \hyperref[theoremA]{A}}
    Let $H$ be a semi-admissible Hamiltonian with $H(r,x)=rT - B$, for $ r \geq r_{\max}(H) = r_{\max}$ and $T\notin S(\alpha,\Lambda_0 \rightarrow \Lambda_1)$, and consider a sequence $s_n \rightarrow \infty$ such that $s_n T \notin \mathcal{S}(\alpha,\Lambda_0 \rightarrow \Lambda_1)$.
    For a Lagrangian tomograph $(L_s)_{s \in B^d}$  of $L_1$ conical in $[r_{\max},\infty)\times \Sigma$, where $B^d = B^d(r) \subset \mathbb{R}^d$, take  
    $$N_n(s) = |\varphi^{s_n}_{H}(\widehat{L_0}) \cap L_s \cap M_{r_{\max}}|.$$
    From the crofton's inequalities follows that
    $$ \int_{B^{d}} N_{n}(s)ds \leq const \cdot \text{vol}(\varphi^{s_n}_{H}(\widehat{L_0}) \cap M_{r_{\max}}),$$
    where $ds$ is the Lebesgue measure in $B^d$. Assuming that the Lagrangians in the tomograph are close enough to $L_1$ in the Hofer norm, or more precisely $d_H(\widehat{L_1},L_s)<\frac{\epsilon}{2}$, then
    \begin{equation} \label{eq: comparing barcodes}
        b_{2\epsilon}(s_nH,L_0 \rightarrow L_1) \leq  b_{\epsilon}(s_nH,L_0 \rightarrow L_s) \leq  N_{n}(s),
    \end{equation}
    where the first and second inequalities follow from \eqref{barcodeofpertubation} and Remark \ref{intersectionoboundingnottooshortbars} respectively. 
    
    Combining
    $$ b_{2\epsilon}(tru(s_nH,L_0 \rightarrow L_1,s_nB)) \leq b_{2\epsilon}(s_nH,L_0 \rightarrow L_1),$$
    and \eqref{eq: comparing barcodes}, we get that
    $$b_{2\epsilon}(s_nH, L_0 \rightarrow L_1,s_nB) \leq const \cdot \mathrm{vol}(\varphi^{s_n}_H(\widehat{L_0}) \cap M_{r_{\max}}),$$
    and by comparing the exponential growth with respect to $n$, we obtain
\begin{align*}
   \limsup_{n\rightarrow \infty} \frac{\log^+(b_{2\epsilon}(tru(s_nH,L_0 \rightarrow L_1,s_nB)))}{s_n} 
   & \leq  \limsup_{n \rightarrow \infty} \frac{\log^+( \text{vol}(\varphi^{s_n}_{H}(\widehat{L_0}) \cap M_{r_{\max}}))}{s_n} \\
   & \leq \tope{\varphi_{H}|_{M_{r_{\max}}}} \\ 
   & = T\tope{\alpha},
\end{align*}
where the second inequality above follows from Yomdin's theorem. Therefore
$$\hbar^{HW}_{2\epsilon}(H,L_0 \rightarrow L_1) \leq \tope{\alpha}.$$
By sending $\epsilon \rightarrow 0$ and Theorem \ref{sequecialbarcodeentropy}, we conclude \eqref{equationentropies}.
\hfill\qed

\bibliographystyle{alpha}\bibliography{mybibliography}

\newcommand{\etalchar}[1]{$^{#1}$}
\begin{thebibliography}{CDSGO16}

\bibitem[Alv16]{alves2016cylindrical}
Marcelo Alves.
\newblock {Cylindrical contact homology and topological entropy}.
\newblock {\em Geometry \& Topology}, \textbf{20}(6):3519--3569, 2016.

\bibitem[Arn90a]{arnold1990dynamics}
VI~Arnold.
\newblock {Dynamics of complexity of intersections}.
\newblock {\em Boletim da Sociedade Brasileira de Matem{\'a}tica-Bulletin/Brazilian Mathematical Society}, \textbf{21}(1):1--10, 1990.

\bibitem[Arn90b]{arnol1990dynamics}
VI~Arnol'd.
\newblock {Dynamics of intersections}.
\newblock In {\em Analysis, et cetera}, pages 77--84. Elsevier, 1990.

\bibitem[BL22]{bae2022comparison}
Hanwool Bae and Sangjin Lee.
\newblock {A comparison of categorical and topological entropies on Weinstein manifolds}.
\newblock {\em arXiv preprint arXiv:2208.14597}, 2022.

\bibitem[BPP{\etalchar{+}}22]{buhovsky2022coarse}
Lev Buhovsky, Jordan Payette, Iosif Polterovich, Leonid Polterovich, Egor Shelukhin, and Vuka{\v{s}}in Stojisavljevi{\'c}.
\newblock {Coarse nodal count and topological persistence}.
\newblock {\em arXiv preprint arXiv:2206.06347}, 2022.

\bibitem[CDSGO16]{chazal2016structure}
Fr{\'e}d{\'e}ric Chazal, Vin De~Silva, Marc Glisse, and Steve Oudot.
\newblock {\em {The structure and stability of persistence modules}}, volume~10.
\newblock Springer, 2016.

\bibitem[CGG21]{cineli2021topological}
Erman Cineli, Viktor~L Ginzburg, and Basak~Z Gurel.
\newblock {Topological entropy of Hamiltonian diffeomorphisms: a persistence homology and Floer theory perspective}.
\newblock {\em arXiv preprint arXiv:2111.03983}, 2021.

\bibitem[CGGM23]{cineli2023invariant}
Erman Cineli, Viktor~L Ginzburg, Basak~Z Gurel, and Marco Mazzucchelli.
\newblock {Invariant sets and hyperbolic closed Reeb orbits}.
\newblock {\em arXiv preprint arXiv:2309.04576}, 2023.

\bibitem[CGGM24]{cineli2024barcode}
Erman Cineli, Viktor~L Ginzburg, Basak~Z Gurel, and Marco Mazzucchelli.
\newblock {On the barcode entropy of Reeb flows}.
\newblock {\em arXiv preprint arXiv:2401.01421}, 2024.

\bibitem[Cin23]{cineli2023generalized}
Erman Cineli.
\newblock {A generalized pseudo-rotation with positive topological entropy}.
\newblock {\em arXiv preprint arXiv:2310.14761}, 2023.

\bibitem[Din71]{dinaburg1971relations}
Efim~I Dinaburg.
\newblock {On the relations among various entropy characteristics of dynamical systems}.
\newblock {\em Mathematics of the USSR-Izvestiya}, \textbf{5}(2):337, 1971.

\bibitem[FLS23]{fender2023barcode}
Elijah Fender, Sangjin Lee, and Beomjun Sohn.
\newblock {Barcode entropy for Reeb flows on contact manifolds with Liouville fillings}.
\newblock {\em arXiv preprint arXiv:2305.04770}, 2023.

\bibitem[FS06]{frauenfelder2006fiberwise}
Urs Frauenfelder and Felix Schlenk.
\newblock {Fiberwise volume growth via Lagrangian intersections}.
\newblock 2006.

\bibitem[GGM22]{ginzburg2022barcode}
Viktor~L Ginzburg, Basak~Z Gurel, and Marco Mazzucchelli.
\newblock {Barcode entropy of geodesic flows}.
\newblock {\em arXiv preprint arXiv:2212.00943}, 2022.

\bibitem[Kat80]{katok1980lyapunov}
Anatole Katok.
\newblock {Lyapunov exponents, entropy and periodic orbits for diffeomorphisms}.
\newblock {\em Publications Math{\'e}matiques de l'IH{\'E}S}, \textbf{51}:137--173, 1980.

\bibitem[Kat82]{katok1982entropy}
Anatole Katok.
\newblock {Entropy and closed geodesies}.
\newblock {\em Ergodic theory and dynamical Systems}, \textbf{2}(3-4):339--365, 1982.

\bibitem[LRSV21]{le2021barcodes}
Fr{\'e}d{\'e}ric Le~Roux, Sobhan Seyfaddini, and Claude Viterbo.
\newblock {Barcodes and area-preserving homeomorphisms}.
\newblock {\em Geometry \& Topology}, \textbf{25}(6):2713--2825, 2021.

\bibitem[Mai22]{mailhot2022spectral}
Pierre-Alexandre Mailhot.
\newblock {The spectral diameter of a Liouville domain}.
\newblock {\em arXiv preprint arXiv:2205.04618}, 2022.

\bibitem[Mei18]{meiwes2018rabinowitz}
Matthias Meiwes.
\newblock {\em {Rabinowitz Floer homology, leafwise intersections, and topological entropy}}.
\newblock PhD thesis, 2018.

\bibitem[MS11]{macarini2011positive}
Leonardo Macarini and Felix Schlenk.
\newblock {Positive topological entropy of Reeb flows on spherizations}.
\newblock In {\em Mathematical Proceedings of the Cambridge Philosophical Society}, volume \textbf{151}, pages 103--128. Cambridge University Press, 2011.

\bibitem[MS12]{mcduff2012j}
Dusa McDuff and Dietmar Salamon.
\newblock {\em {J-holomorphic curves and symplectic topology}}, volume~52.
\newblock American Mathematical Soc., 2012.

\bibitem[Pat12]{paternain2012geodesic}
Gabriel~P Paternain.
\newblock {\em {Geodesic flows}}, volume 180.
\newblock Springer Science \& Business Media, 2012.

\bibitem[PRSZ20]{polterovich2020topological}
Leonid Polterovich, Daniel Rosen, Karina Samvelyan, and Jun Zhang.
\newblock {\em {Topological persistence in geometry and analysis}}, volume~74.
\newblock American Mathematical Soc., 2020.

\bibitem[Rit09]{ritter2009novikov}
Alexander~F Ritter.
\newblock {Novikov-symplectic cohomology and exact Lagrangian embeddings}.
\newblock {\em Geometry \& Topology}, \textbf{13}(2):943--978, 2009.

\bibitem[Rit13]{ritter2013topological}
Alexander~F Ritter.
\newblock {Topological quantum field theory structure on symplectic cohomology}.
\newblock {\em Journal of Topology}, \textbf{6}(2):391--489, 2013.

\bibitem[Sal99]{salamon1999lectures}
Dietmar Salamon.
\newblock {Lectures on Floer homology}.
\newblock {\em Symplectic geometry and topology (Park City, UT, 1997)}, \textbf{7}:143--229, 1999.

\bibitem[UZ16]{usher2016persistent}
Michael Usher and Jun Zhang.
\newblock {Persistent homology and Floer--Novikov theory}.
\newblock {\em Geometry \& Topology}, \textbf{20}(6):3333--3430, 2016.

\bibitem[Yom87]{yomdin1987volume}
Yosef Yomdin.
\newblock {Volume growth and entropy}.
\newblock {\em Israel Journal of Mathematics}, \textbf{57}:285--300, 1987.

\end{thebibliography}
\end{document}